\documentclass[12pt]{amsart}
\usepackage{fullpage,graphicx,amsfonts,amssymb,amsmath,amsthm,xcolor,float,subcaption}
\usepackage[all]{xy}
\usepackage[export]{adjustbox} 

\theoremstyle{plain} 
\newtheorem{theorem}    {Theorem}
\newtheorem{theoremletter}{Theorem}

\newtheorem{conjectureletter}[theoremletter]{Conjecture}
\newtheorem{lemma}      [theorem]{Lemma}
\newtheorem{corollary}  [theorem]{Corollary}
\newtheorem{proposition}[theorem]{Proposition}

\newtheorem{conjecture} [theorem]{Conjecture}

\theoremstyle{definition}

\theoremstyle{remark}
\newtheorem{remark}              {Remark}

\numberwithin{equation}{section}

\newenvironment{spmatrix}{\left(\begin{smallmatrix}}{\end{smallmatrix}\right)}

\newcommand{\w}{\operatorname{w}}
\newcommand{\BP}{\operatorname{BP}}
\newcommand{\Z}{\mathbb Z}

\newcommand{\h}{\mathfrak{H}}
\newcommand{\SL}{\operatorname{SL}}
\newcommand{\Res}{\operatorname{Res}}
\renewcommand{\Re}{\operatorname{Re}}
\renewcommand{\Im}{\operatorname{Im}}

\title{The variation of zeros of the Miller basis}
\author{Liubomir Chiriac\and Andrei Jorza}
\date{}

\begin{document}
\maketitle
\begin{abstract}
We exhibit a connection between the variation of zeros in the Miller basis of modular forms $q^m+O(q^{\ell+1})$ and a logarithmic version $\mathcal{S}_\delta$ of the Szeg\H{o} curve, where $\delta=m/\ell$. When $\delta<0.6194$ we show that all the zeros are on the unit arc for $k\gg 0$, while if $\delta$ is asymptotically close to 1, we show that all the zeros lie on $\mathcal{S}_{\delta}$. In general, we posit that for all $\delta$, the zeros are located on the union of the unit arc and the log Szeg\H{o} curve, obtaining a partial result, and find conjectural thresholds for $m/\ell$ with all zeros on the unit arc, and no zeros on the arc. Finally, we enumerate all algebraic zeros of Miller forms up to $\ell-m\leq 25$. 
\end{abstract}

\section{Introduction}

Denote by $M_k$ the space of all holomorphic modular forms of even weight $k\ge 4$ and level one on the upper half-plane $\h$. For any element of $M_k$ we consider its zeros in the standard fundamental domain $\mathcal{F}$. The study of the location of these zeros was revolutionized by Rankin and Swinnerton-Dyer \cite{rankin-swinnerton-dyer}, who proved that all non-elliptic zeros, i.e., distinct from $i$ and $\rho=e^{2\pi i/3}$, of the Eisenstein series \[ E_k(z)=\frac{1}{2} \sum_{ (c,d)=1} \frac{1}{(cz+d)^k}\] in $\mathcal{F}$ lie on the arc $\mathcal{A}$ of the unit circle joining $i$ and $\rho$. Their elegant approach, which involves approximating $E_k$ by an elementary function on $\mathcal{A}$, continues to serve as a blueprint for most subsequent results on the subject.

Building on this foundational idea, Duke and Jenkins \cite{duke-jenkins} incorporated the circle method as a crucial new component to show that the zeros of certain weakly holomorphic modular forms also lie on the unit circle. Raveh \cite{raveh:miller} recently adapted their proof to study the zeros of certain elements of the Miller basis of $M_k$ denoted by 
\begin{equation}\label{def-Miller-basis}
g_{k,m}=q^m+O\left( q^{\ell +1} \right),\text{ where } \ell+1=\dim M_k, 0\le m\le \ell.
\end{equation}

He showed that for fixed $m\ge 1$, once $\ell$ exceeds a certain linear bound in $m$, the zeros of $g_{k,m}$
lie on $\mathcal{A}$, a property that always holds true for $m=1$. In contrast, when the difference $D=\ell-m$ is fixed, Rudnick \cite{rudnick:faber} used Faber polynomials to show that the zeros of $g_{k,m}$ cluster near vertical lines as $k\to \infty$. 

One of the main motivations for this paper is to provide a framework that encompasses the entire family of Miller basis. Our first result provides a lower-bound for the proportion of non-elliptic zeros on the arc depending on the ratio $\delta=m/\ell$. 

\begin{theoremletter}(Theorem \ref{t:main-arc})\label{tl:main-arc}
As $k\to \infty$ we have that
\[\frac{1}{D}
\left|
\left\{
\parbox{3.3cm}{
\centering
non-elliptic zeros\\
of $g_{k,m}$ on $\mathcal{A}$
}
\right\}
\right|\geq \begin{cases}1&\delta<0.6194\\ 1-2.9832(\delta-0.6194)& 0.6194\leq \delta<0.9546\\ 0&\delta\geq 0.9546.\end{cases}\]

\end{theoremletter}

In particular, for sufficiently large weights $k$, all the zeros of $g_{k,m}$ in the fundamental domain lie on the arc $\mathcal{A}$ whenever $\delta<61.94\%$, and $g_{k,m}$ has 
at least one zero on $\mathcal{A}$ as long as $\delta<95.46\%$. These theoretical thresholds align closely with our numerical experiments, which we ran for many $k$ up to $600,000$. A sample of the data is shown in the following table.

\begin{center}
\begin{tabular}{llll}
$k$ & $\ell$ & Not all roots on $\mathcal{A}$ & No roots on
                                 $\mathcal{A}$\\
  \hline
$18000$ & $1500$ & $m\geq 936$ & $m\geq 1433$\\
$18002$ & $1500$ & $m\geq 947$ & $m\geq 1421$\\
$18004$ & $1500$ & $m\geq 942$ & $m\geq 1434$\\
$18006$ & $1500$ & $m\geq 937$ & $m\geq 1421$\\
$18008$ & $1500$ & $m\geq 946$ & $m\geq 1434$\\
$18010$ & $1500$ & $m\geq 943$ & $m\geq 1421$
\end{tabular}
\end{center}

In Section~\ref{Weakly-holom-Miller} we derive a parallel result for the Miller basis of the space of weakly modular forms, complementing the work of Duke and Jenkins.

A central insight of Rudnick's result, as discussed above, is that one can estimate the coefficients of the Faber polynomials associated to the Miller basis when $D$ is fixed. In section 5, we extend the method by allowing $D\to \infty$,  where the analysis becomes substantially more delicate and requires the growth of $D$ to remain sufficiently slow relative to $k$. This refinement reveals a new phenomenon in the asymptotic behavior of the zeros: they converge toward a logarithmic analogue of the classical Szeg\H{o} curve \[ \mathcal{S}=\{z\in \mathbb{C}: |z e^{1-z}|=1\}.\]  Its $x$-intercepts occur at $-W(e^{-1})\approx -0.27846$, where $W$ is the Lambert $W$-function, and $1$, while its $y$-intercepts at $\pm e^{-1}$. 

\begin{figure}[H]
\centering
\includegraphics[scale=0.5]{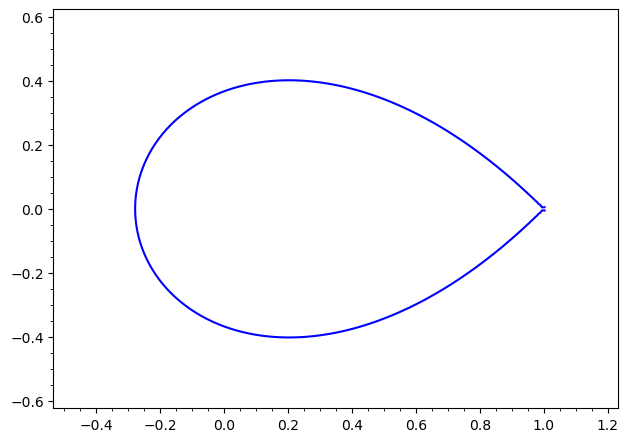}
\caption{The Szeg\H{o} curve}
\end{figure}

The logarithmic Szeg\H{o} curve
$\mathcal{L}_\pm=\{ \tau \in \h: |\Re \tau |\le 1/2, \pm
24e^{2\pi i \tau}\in \mathcal{S}\}$ is the graph of the function
$g_{\pm}:[-0.5,0.5]\to \h$ given by
\[g_\pm(x)=\frac{\ln 24 - \ln u^{-1}(\pm \cos(2\pi x))}{2\pi},\] where $u:[W(e^{-1}),1]\to [-1,1]$ is the bijection $u(x)=(1+\ln x)/x$. 

\begin{theoremletter}(Theorem \ref{t:conj-tinyD}) \label{tl:conj-tinyD}
Suppose $D\to\infty$ such that $D<\frac{\alpha \log |k|}{\log\log |k|}$ for some $\alpha\in (0,1)$.
Then the zeros of $g_{k,m}$ asymptotically approach
$\mathcal{S}_\delta=\mathcal{L}_\pm -\frac{1}{2\pi}\log|1-\delta|$,
where $\delta=\frac{m}{\ell}$ and $\pm$ is the sign of $k$.

\begin{figure}[H]
     \centering
     \begin{subfigure}[b]{0.45\textwidth}
         \centering
         \includegraphics[width=\textwidth]{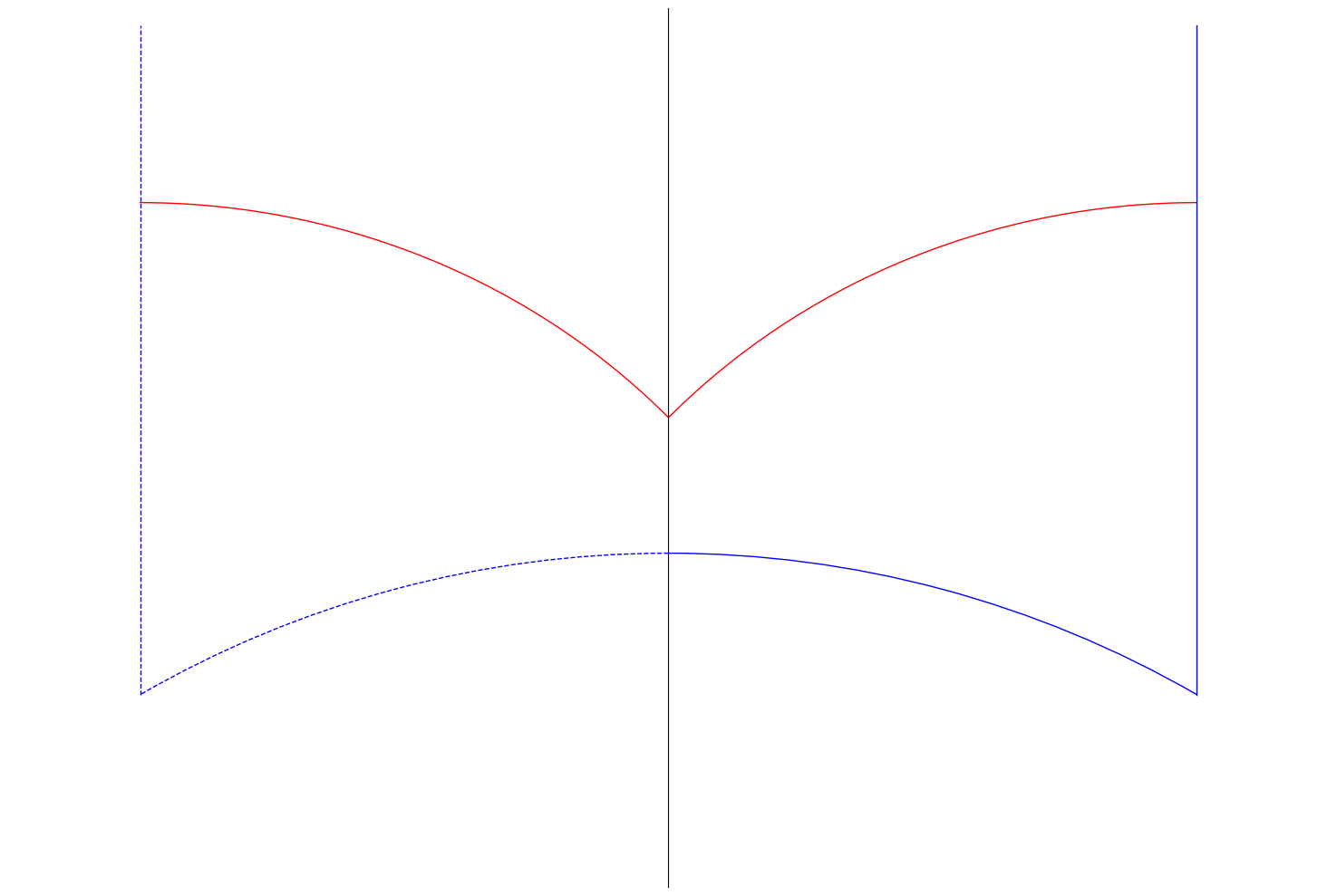}
         \subcaption{$\mathcal{S}_\delta$ when $\delta=0.98$}
         \label{fig:positive_delta}
     \end{subfigure}
     \hfill
     \begin{subfigure}[b]{0.45\textwidth}
         \centering
         \includegraphics[width=\textwidth]{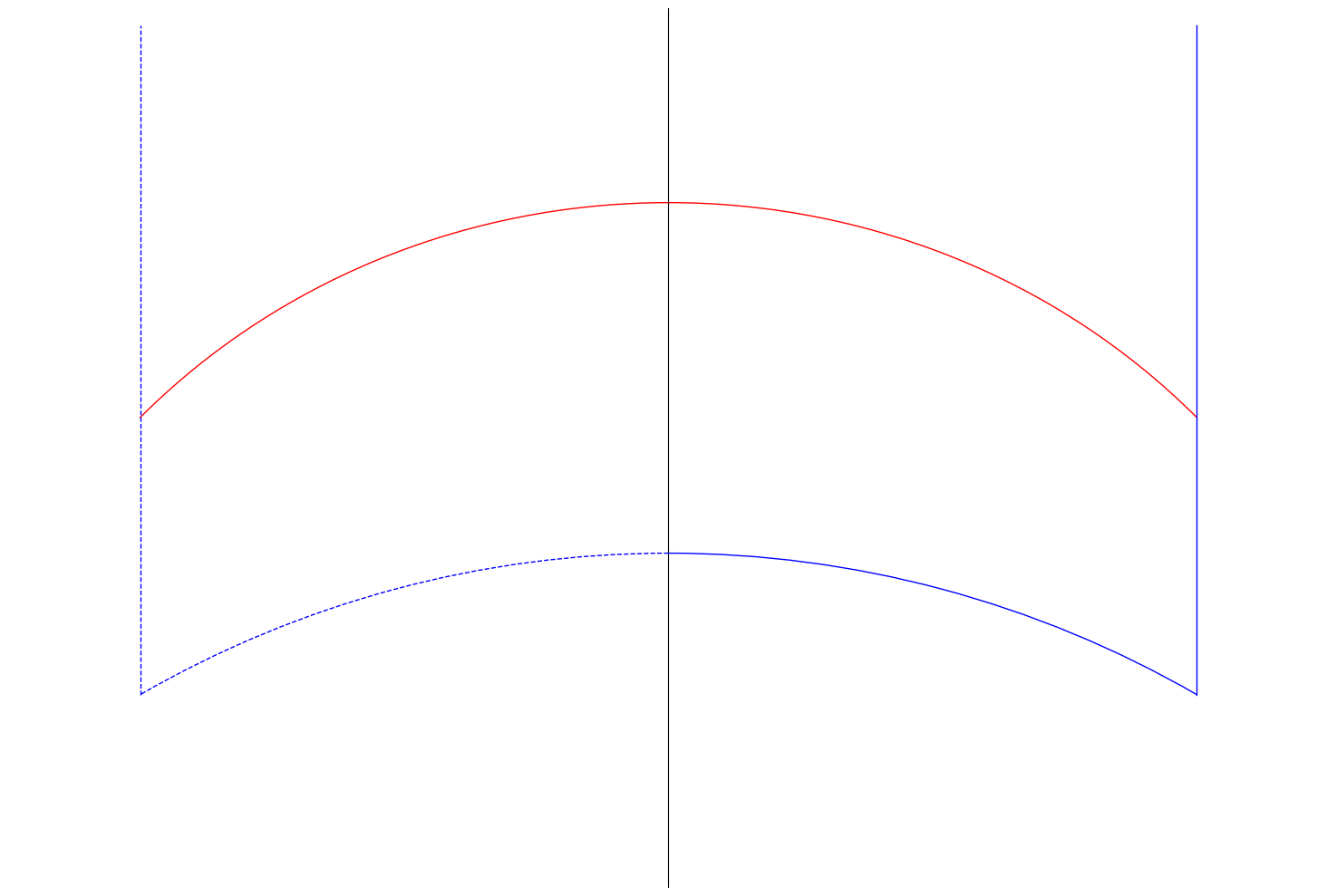}
         \subcaption{$\mathcal{S}_\delta$ when $\delta=1.02$}
         \label{fig:negative_delta}
     \end{subfigure}
     \label{fig:szego_comparison}
     \caption{Asymptotic location of zeros as $\delta\to 1$}
\end{figure}

\end{theoremletter}

The Szeg\H{o} curve arises naturally as the limiting locus of the zeros of truncated exponential polynomials under an appropriate normalization. A noteworthy aspect of our theorem is that the relevant truncation threshold is governed solely by the ratio $m/\ell$, as in the preceding result. This observation leads us to formulate the following general conjecture.

\begin{conjectureletter}(Conjecture \ref{conj}) \label{cl:conj}
For each $\delta>0$, the zeros of $g_{k,m}$ asymptotically
approach the curve $\mathcal{C}_\delta$, defined as the upper
hull of the union of $\mathcal{A}$ and $\mathcal{S}_\delta$.
\end{conjectureletter}

The cutoffs of $0.6194$ and $0.9546$ in Theorem
\ref{tl:main-arc}, as well as $1.1026$ and $1.1598$ for weakly holomorphic modular forms from Section \ref{Weakly-holom-Miller}, can be recontextualized in light of
Conjecture \ref{cl:conj}, as approximations of the conjectural
cutoffs (see Section \ref{Conject}, particularly Remark \ref{r:conjcutoffs}):
\begin{align*}
  \delta_{\mathcal{A}}^+&=1-\frac{24}{W(e^{-1})e^{\sqrt{3}\pi}}=0.6265\ldots&\delta_{\mathcal{A}}^-&=1+\frac{24}{W(e^{-1})e^{2\pi}}=1.1609\ldots\\
  \delta_{\mathcal{S}}^+&=1-\frac{24}{e^{2\pi}}=0.9551\ldots&\delta_{\mathcal{S}}^-&=1+\frac{24}{e^{\sqrt{3}\pi}}=1.1040\ldots
\end{align*}

Theorems \ref{tl:main-arc} and \ref{tl:conj-tinyD} prove parts of Conjecture \ref{cl:conj}. In general, we have the following partial result:
\begin{theoremletter}[Theorem \ref{t:logk}]
If $k\to \infty$ then all the zeros of the Miller modular forms $g_{k,m}$ lie in the region $\{z\in \mathcal{F}\mid \Im z < \frac{2\log k}{1-c}+O(1)\}$, where $c=\frac{e^{2\pi}}{1728}$.
\end{theoremletter}

Beyond their geometric location, the arithmetic nature of these zeros is also of significant interest. Kohnen \cite{kohnen:eisenstein} proved that all non-elliptic zeros of Eisenstein series are transcendental. Alongside a theorem of Schneider and classical results from the theory of complex multiplication, the fact that all such zeros lie on the arc $\mathcal{A}$ plays a decisive role in his analysis. In Section~\ref{Algebraic-zeros}, we conclude this paper by contrasting this phenomenon with the Miller basis, exhibiting forms $g_{k,m}$ that possess algebraic zeros other than $\rho$ or $i$. More precisely, we determine all such forms in the range $\ell-m\le 25$. 

\section{Approximating the Miller forms}

If $k\ge 4$ is an even weight and $\ell=\dim M_k -1$, we can write 
$ k=12\ell +k'$ for some \[ k'\in \{0, 4,6, 8, 10, 14\}.\]
Consider the Miller basis of $M_k$ given by $g_{k,m}=q^m+O(q^{\ell+1}).$
Since each $g_{k,m}$ has integer Fourier coefficients, it is clear that $\overline{g_{k,m}(z)}=g_{k,m}(-\overline{z}).$ Setting $z=e^{i\theta}$ we obtain
\[\overline{g_{k,m}(e^{i\theta})}=g_{k,m}(-e^{-i\theta})=g_{k,m}(\begin{spmatrix}
&-1\\1&\end{spmatrix}e^{i\theta})=
e^{ik\theta}g_{k,m}(e^{i\theta}).\]
Therefore $g_{k,m}$ can be rescaled along the arc to obtain the
real-valued function $e^{ik\theta/2}g_{k,m}(e^{i\theta})$. We denote
\[\overline{g}_{k,m}(e^{i\theta}) = e^{ik\theta/2+2\pi m\sin\theta} g_{k,m}(e^{i\theta}).\]

One can show (see, for example, \cite[Lemma 3.7]{raveh:miller}) that for each $R>0$ there exists an $A>1$ such that for all $z\in \mathcal{F}$ with $|j(z)|<R$ we have
\[g_{k,m}(z)=\int_{-\frac{1}{2}+iA}^{\frac{1}{2}+iA}G(z,\tau)d\tau,\] where
\[G(z,\tau)=\frac{\Delta^\ell(z)E_{k'}(z)E_{14-k'}(\tau)}{\Delta^{\ell+1}(\tau)(j(\tau)-j(z))}e^{2\pi
  i m \tau}.\]
Since we are primarily interested in roots on $\mathcal{A}$, we may take $R>1728$ and a corresponding $A$, so that the above integral formula holds. Duke-Jenkins and Raveh use the circle method to find good
estimates of the above integral. Abstracting their process, for
each $B>0$, we will integrate on the following contour
\begin{figure}[H]
\centering
\includegraphics[scale=0.25]{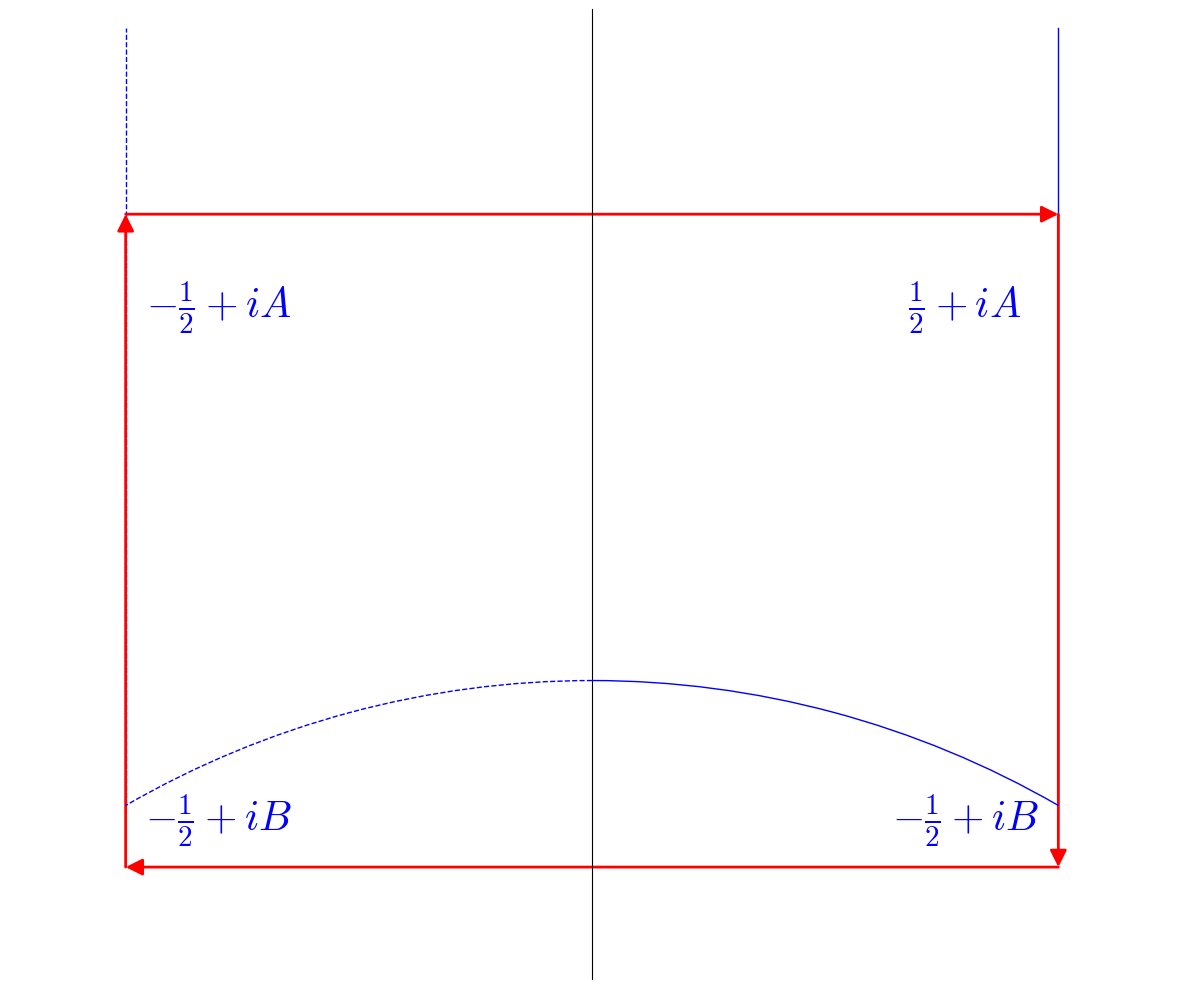}
\caption{Contour of integration}
\end{figure}
This yields
\begin{equation}\label{eq:G}g_{k,m}(e^{i\theta})=\int_{-\frac{1}{2}+iB}^{\frac{1}{2}+iB}G(\tau,e^{i\theta})d\tau-2\pi
i\sum_{\gamma}\Res_{\tau = \gamma
  e^{i\theta}}G(\tau,e^{i\theta}),\end{equation}
where the sum is over all $\gamma\in \SL_2(\mathbb{Z})$ such
that $\Re \gamma e^{i\theta}=-\frac{1}{2}$ and $\Im \gamma
e^{i\theta}>B$, as well as the identity matrix $I_2$ and $S=\begin{spmatrix}
&-1\\1&\end{spmatrix}$. The main idea is that the residues at $I_2$ and $S$ control the size of $\overline{g}_{k,m}(e^{i\theta})$, the other terms being errors that one can estimate.


\begin{lemma}\label{l:poles}
For any $e^{i\theta}\in \mathcal{F}$, the function $j(\tau)-j(e^{i\theta})$ has finitely many roots inside the contour above. At any root $\gamma e^{i\theta}$, which corresponds to a simple pole of $G(\tau,e^{i\theta})$, we have
\[\Res_{\tau = \gamma
  e^{i\theta}}G(\tau,e^{i\theta})=-\frac{1}{2\pi i} (ce^{i\theta}+d)^{-k}
e^{2\pi i m (\gamma e^{i\theta})},\]
giving
\[|-2\pi i\Res_{\tau = \gamma
  e^{i\theta}}G(\tau,e^{i\theta})|=|ce^{i\theta}+d|^{-k}
e^{-2\pi  m \sin\theta|ce^{i\theta}+d|^{-2}}.\]
\end{lemma}
\begin{proof}
Indeed,
\begin{align*}
\Res_{\tau = \gamma
  e^{i\theta}}G(\tau,e^{i\theta})&=\frac{\Delta^\ell(e^{i\theta})E_{k'}(e^{i\theta})E_{14-k'}(\gamma e^{i\theta})}{\Delta^{\ell+1}(\gamma e^{i\theta})\frac{dj}{d\tau}(\gamma e^{i\theta})}e^{2\pi
                                   i m \gamma e^{i\theta}}\\
  &=(c e^{i\theta}+d)^{-k}\frac{\Delta^\ell(\gamma e^{i\theta})E_{k'}(\gamma e^{i\theta})E_{14-k'}(\gamma e^{i\theta})}{\Delta^{\ell+1}(\gamma e^{i\theta})\frac{dj}{d\tau}(\gamma e^{i\theta})}e^{2\pi
    i m \gamma e^{i\theta}}\\
  &=-\frac{1}{2\pi i} (ce^{i\theta}+d)^{-k}
e^{2\pi i m (\gamma e^{i\theta})},
\end{align*}
as $E_{k'}(\tau)E_{14-k'}(\tau)=E_{14}(\tau)$ and $\frac{dj}{d\tau}=-2\pi i \frac{E_{14}(\tau)}{\Delta(\tau)}$.
\end{proof}

In particular, the residues at $z$ and $\begin{spmatrix}
&-1\\1&\end{spmatrix}z$ have a combined contribution of
\begin{align*}
-2\pi i \left(\Res_{\tau=e^{i\theta}}+\Res_{\tau=-e^{-i\theta}}\right) &= e^{2\pi i m e^{i\theta}}+ e^{-ik\theta}e^{-2\pi i m e^{-i\theta}} \notag \\
&=e^{-2\pi m\sin\theta-ik\theta/2} 2\cos\left(2\pi m\cos (\theta)+k\theta/2\right).
\end{align*}

The technical core of this section, and of the zeros counting argument in the next, relies on estimating each of the remaining terms in \eqref{eq:G}. First, it is clear that the error satisfies
\begin{align*}
E &=|\overline{g}_{k,m}(e^{i\theta})-2\cos\left(2\pi m\cos
(\theta)+k\theta/2\right)| \notag \\
& \leq e^{2\pi m\sin\theta}\int_{-\frac{1}{2}+iB}^{\frac{1}{2}+iB}|G(\tau,e^{i\theta})|d\tau+e^{2\pi m\sin\theta}\sum_{\gamma\neq I_2,S}|ce^{i\theta}+d|^{-k}
e^{-2\pi  m \sin\theta|ce^{i\theta}+d|^{-2}}.
\end{align*}

\begin{lemma}\label{l:IB}
Let $\frac{\pi}{2}\leq \alpha\leq \beta \leq\frac{2\pi}{3}$. 
For each $B>0$ such that $\Im \gamma e^{i\theta}\neq B$ for any $\gamma$ as above and all $\theta\in [\alpha,\beta]$, there is a constant $h_{\alpha,\beta,B}$ such that
\[\int_{-\frac{1}{2}+iB}^{\frac{1}{2}+iB}|G(\tau,e^{i\theta})|d\tau\leq e^{-2\pi m B+h_{\alpha,\beta,B}}\left(\frac{|\Delta(e^{i\theta})|}{|\Delta(iB)|}\right)^\ell .\]
\end{lemma}
\begin{proof}
We follow Duke-Jenkins in estimating the integral:

\[\int_{-\frac{1}{2}+iB}^{\frac{1}{2}+iB}|G(\tau,e^{i\theta})|d\tau\leq e^{-2\pi m B}\max_{|x|\leq \frac{1}{2}}\left(\frac{|\Delta(e^{i\theta})|}{|\Delta(x+iB)|}\right)^\ell \int_{-\frac{1}{2}+iB}^{\frac{1}{2}+iB}\frac{|E_{k'}(e^{i\theta})E_{14-k'}(\tau)|}{|\Delta(\tau)(j(\tau)-j(e^{i\theta})|}d\tau.\]

Remark that $\min_{|x|\leq
  \frac{1}{2}}|\Delta(x+iB)|=|\Delta(iB)|$. Indeed, the
$q$-expansion of $\frac{q}{\Delta}=\sum c_nq^n$ has nonnegative
coefficients $c_n\geq 0$. Plugging in $z=x+iB$, so $|q|=e^{-2\pi B}$, we get: 
\[\frac{e^{-2\pi B}}{|\Delta(x+iB)|}\leq \sum c_n e^{-2\pi
  nB}=\frac{e^{-2\pi B}}{|\Delta(iB)|}.\]
Finally, the function $f(\theta,\tau)=\frac{|E_{k'}(e^{i\theta})E_{14-k'}(\tau)|}{|\Delta(\tau)(j(\tau)-j(e^{i\theta})|}$ is analytic on the compact set $[\sigma,\beta]\times \{x+iB\mid x\in [-\frac{1}{2},\frac{1}{2}]\}$ and thus achieves a maximum $e^{h_{\alpha,\beta,B}}$.
\end{proof}

\begin{lemma}\label{l:miller-approx-arc}
Let $\alpha, \beta$ as above and $B\in (0,\sin \alpha)$.
There exists a constant $I(\alpha,\beta,B)$ such that 
for all $\theta\in [\alpha,\beta]$ we have
\[\overline{g}_{k,m}(e^{i \theta})=2\cos\left(2\pi m\cos
(\theta)+k\theta/2\right)+o(1),\]
as $k\to \infty$ and $\frac{m}{\ell}<I(\alpha,\beta,B)$.
\end{lemma}
\begin{proof}
Let $\mathcal{R}$ be the finite set of elements $\gamma\in \SL_2(\mathbb{Z})$ such that $\Im \gamma e^{i\theta}\geq B$ for some $\theta\in [\alpha,\beta]$, excluding $I_2$ and $S$.
From Lemma \ref{l:IB} we conclude that 
\[E\leq e^{2\pi m (\sin\theta-B)+h_B -\ell
  (\ln|\Delta(iB)|-\ln|\Delta(e^{i\theta})|)}+\sum_{\gamma\in \mathcal{R}} e^{2\pi
  m\sin\theta(1-|ce^{i\theta}+d|^{-2})-k\ln|ce^{i\theta}+d|}.\]
Writing $\delta=\frac{m}{\ell}\in [0,1]$ we get
\[E\leq e^{\ell(2\pi \delta (\sin\theta-B) -
  (\ln|\Delta(iB)|-\ln|\Delta(e^{i\theta})|))+h_B}+\sum_{\gamma\in \mathcal{R}} e^{\ell(2\pi
  \delta \sin\theta(1-|ce^{i\theta}+d|^{-2})-12\ln|ce^{i\theta}+d|)-k'\ln|ce^{i\theta}+d|}.\]

When $k\to\infty$, the error term above approaches 0 as long as
each of the following terms is negative:
\begin{align}\label{eq:upperbound}
2\pi \delta (\sin\theta-B) -
  (\ln|\Delta(iB)|-\ln|\Delta(e^{i\theta})|)&<0\\
  2\pi
  \delta \sin\theta(1-|ce^{i\theta}+d|^{-2})-12\ln|ce^{i\theta}+d|&<0\ \ \ \forall \gamma\in \mathcal{R}.\notag
\end{align}
The bounds \eqref{eq:upperbound} are guaranteed for
all such $\theta\in [\alpha,\beta]$ if
\begin{align}\label{eq:upperboundab}
2\pi \delta (\sin\alpha-B) -
  (\ln|\Delta(iB)|-\ln|\Delta(e^{i\beta})|)&<0\\
  2\pi
  \delta\sin\alpha(1-|ce^{i\alpha}+d|^{-2})-12\ln|ce^{i\beta}+d|&<0\ \ \ \forall \gamma\in \mathcal{R}\ \ cd\geq 0\notag\\
  2\pi
  \delta \sin\alpha(1-|ce^{i\beta}+d|^{-2})-12\ln|ce^{i\alpha}+d|&<0\ \ \ \forall \gamma\in \mathcal{R}\ \ cd<0.\notag
\end{align}
Indeed, we use that $|\Delta(e^{i\theta})|$ is increasing on
$[\alpha,\beta]$ (\cite[Prop 3.1]{raveh:miller}) and that
$|ce^{i\theta}+d|^2=c^2+d^2+2cd\cos\theta$ is $\geq 1$ on $[\frac{\pi}{2},\frac{2\pi}{3}]$, and is increasing if and
only if $cd<0$. Finally, consider $I(\alpha, \beta, B)$ the
minimum of the following:
\begin{itemize}
\item $\dfrac{\ln|\Delta(iB)|-\ln|\Delta(e^{i\beta})|}{2\pi
  (\sin\alpha-B)}$
\item $\dfrac{12\ln|ce^{i\beta}+d|}{2\pi\sin \alpha
  (1-|ce^{i\alpha}+d|^{-2})}$ for all $cd\geq 0$ and
\item $\dfrac{12\ln|ce^{i\alpha}+d|}{2\pi\sin \alpha (1-|ce^{i\beta}+d|^{-2})}$ for all $cd< 0$.
\end{itemize}
\end{proof}
Thus, for each $\theta$ one can obtain a set of upper bounds for
$\delta=\frac{m}{\ell}$ which guarantee the error $E\to 0$ as
$k\to \infty$. One may reinterpret the approach of Duke-Jenkins and Raveh as bounding $I(\frac{\pi}{2}, 1.9, 0.75)$ and $I(1.9,\frac{2\pi}{3}, 0.65)$. 
Our method relies on the fact that different values of $B>0$ are better suited for different ranges of $\theta$, which we will execute using Sage Math
\cite{sage} in the following section.

\section{Proof of Theorem A}
The approximation from Lemma \ref{l:miller-approx-arc} provides the
control needed to count zeros of $g_{k,m}$ on any portion of the arc $\mathcal{A}$. 

\begin{proposition}\label{p:arc-count}
Suppse $\frac{\pi}{2}\leq \alpha<\beta\leq \frac{2\pi}{3}$ and
$B<\sin \alpha$. As $k\to\infty$ such that
$\delta=\frac{m}{\ell}<\min(I(\alpha,\beta,B),\frac{3}{\pi})$, the Miller form
$g_{k,m}$ at least
\[\left\lfloor\frac{k \beta}{2\pi}+2 m \cos
\beta\right\rfloor - \left\lceil\frac{k \alpha}{2\pi}+2 m
\cos \alpha \right\rceil\]
roots on $\mathcal{A}_{\alpha,\beta}=\{e^{i\theta}\mid
\alpha\leq \theta\leq \beta\}$.
\end{proposition}

\begin{proof}
Lemma \ref{l:miller-approx-arc} guarantees that, as $k\to\infty$:
\[|\overline{g}_{k,m}(e^{i\theta})-2\cos(k\theta/2+2\pi m\cos\theta)|<2,\]
for all $\theta\in \mathcal{A}_{\alpha,\beta}$ as long as
$\frac{m}{\ell}<I(\alpha,\beta,B)$. When this inequality is satisfied,
as in \cite[Lemma 3]{duke-jenkins} and \cite[\S
3.4]{raveh:miller} we see that $k\theta/2+2\pi m\cos\theta$ is
increasing on $[\alpha,\beta]$ as
$\delta\leq\frac{3}{\pi}$. This means that the expression takes
at least
\[\left\lfloor\frac{k \beta}{2\pi}+2 m \cos
\beta\right\rfloor - \left\lceil\frac{k \alpha}{2\pi}+2 m
\cos \alpha \right\rceil+1\]
consecutive integer multiple of $\pi$ values in $[\alpha,\beta]$, and thus $2\cos(k\theta/2+2\pi m\cos\theta)$ changes sign in $[\alpha,\beta]$ at least $\left\lfloor\frac{k \beta}{2\pi}+2 m \cos
\beta\right\rfloor - \left\lceil\frac{k \alpha}{2\pi}+2 m
\cos \alpha \right\rceil$ times.
\end{proof}

The stage is now set for the proof of our first main result.

\begin{theorem}\label{t:main-arc}
As $k\to \infty$ we have that
\[\frac{1}{D}
\left|
\left\{
\parbox{3.3cm}{
\centering
non-elliptic zeros\\
of $g_{k,m}$ on $\mathcal{A}$
}
\right\}
\right|\geq \begin{cases}1&\delta<0.6194\\ 1-2.9832(\delta-0.6194)& 0.6194\leq \delta<0.9546\\ 0&\delta\geq 0.9546.\end{cases}\]
Moreover, as $k\to \infty$ a subset of the roots of $g_{k,m}$ will become equidistributed on the arc from $\pi/2$ to 
\[\Theta(\delta)=
	\begin{cases}
		2\pi/3 &\delta<0.6194\\
  		2\pi/3- 2.9832(\delta-0.6194)\pi / 6 & 0.6194\leq \delta<0.9546\\
  		\pi/ 2 &\delta\geq 0.9546.
\end{cases}\]

More precisely, there exists a decreasing piece-wise linear function
$\mathcal{P}:[0,1]\to [0,1]$, explicitly computed in Sage
and whose graph is given below, such that for $k$ large enough
\[\frac{1}{D}
\left|
\left\{
\parbox{3.3cm}{
\centering
non-elliptic zeros\\
of $g_{k,m}$ on $\mathcal{A}$
}
\right\}
\right|
\geq
\mathcal{P}\left(\frac{m}{\ell}\right).\]
\end{theorem}

\begin{proof}
We subdivide the interval $[\frac{\pi}{2},\frac{2\pi}{3}]$ into
$N=1000$ equal intervals $I_r=[\frac{\pi}{2}+\frac{\pi r}{6N},
\frac{\pi}{2}+\frac{\pi(r+1)}{6N}]$, for $0\leq r<N$, and approximate the Miller form $g_{k,m}$ on each arc $\mathcal{A}_r = e^{i I_r}$ separately.

On each interval $I_r=[\alpha_r,\beta_r]$ we choose $B_r\in (\frac{1}{2}\tan\frac{\beta_r}{2},\sin \alpha_r)$ and compute
the bound $\delta_r(B_r)=I(\alpha_r,\beta_r,B_r)$. Lemma \ref{l:miller-approx-arc}
guarantees that, as $k\to\infty$,
\[\overline{g}_{k,m}(e^{i\theta})=2\cos(k\theta/2+2\pi
m\cos\theta)+o(1)\]
for all $\theta\in I_r$ as long as $\frac{m}{\ell}<\delta_r(B_r)$ and $\frac{3}{\pi}$. We remark that the lower bound $B_r> \frac{1}{2}\tan\frac{\beta_r}{2}\geq \frac{1}{2}\geq \frac{1}{2}\cot\frac{\alpha_r}{2}$ implies that the set $\mathcal{R}$ (as defined in Lemma  \ref{l:miller-approx-arc}) is empty.

We want to choose $B_r\in (\frac{1}{2}\tan\frac{\beta_r}{2},\sin \alpha_r)$ such that $\delta_r(B_r)$ is as large as possible. We perform an incremental search in Sage: as $B_r$ varies from $\frac{1}{2}\tan\frac{\beta_r}{2}$ to $\sin \alpha_r$ in increments of $0.0005$ we compute $\delta_r(B_r)$ and in the end choose $B_r$ such that $\delta_r=\delta_r(B_r)$ is maximal. The following graph shows the values of $B_r$ maximizing $\delta_r(B_r)$ under each
corresponding arc $\mathcal{A}_r$, as well as the choice of
$B_{\operatorname{DJ}}=0.65$ and $B_{\operatorname{DJ}}=0.75$
from \cite{duke-jenkins}. By inspection, $B_r$ and $\delta_r$ are decreasing in $r$.
\begin{figure}[H]
\centering
\includegraphics[width=0.5\textwidth]{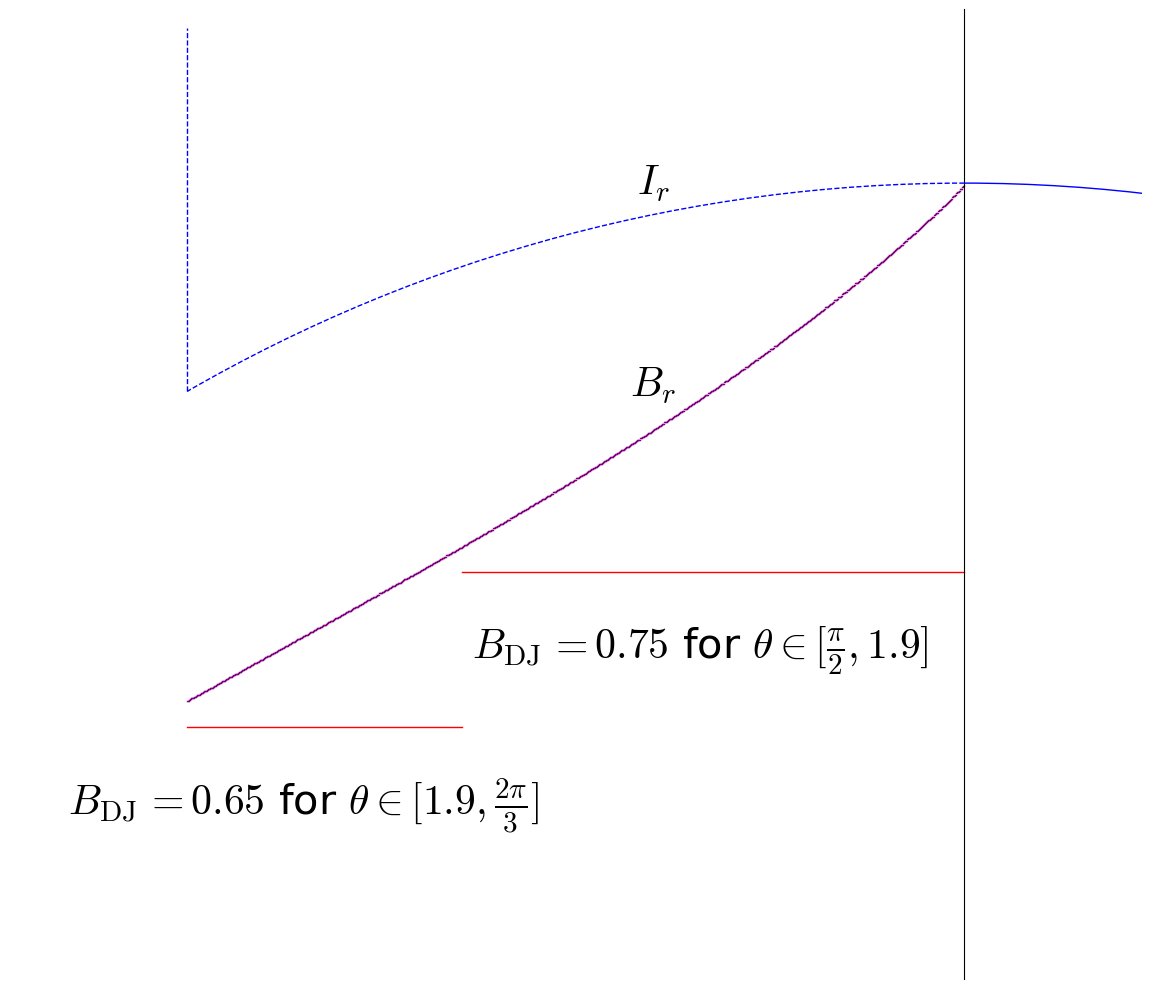}
\caption{Choice of $B_r$ for each arc $\mathcal{A}_r$}
\end{figure}

Suppose, now, that $\delta<\frac{3}{\pi}\approx 0.9549$. Since $\delta_r$ is decreasing in $r$, there exists a maximum $r_0$ such that
$\delta<\delta_{r}$ for all $r\leq r_0$. We conclude that
\[\overline{g}_{k,m}(e^{i\theta})=2\cos(k\theta/2+2\pi
m\cos\theta)+o(1)\]
for all $\theta\in [\frac{\pi}{2},\beta_{r_0}]$. 
Then
Proposition \ref{p:arc-count} applies for this interval and so
the Miller form $g_{k,m}$ will have at least
\[\left\lfloor \frac{k \beta_{r_0}}{2\pi}+2 m\cos \beta_{r_0}\right\rfloor-\left\lceil \frac{k}{4}\right\rceil\]
roots on the arc $\mathcal{A}$, all of them lying on the arc corresponding to the angle interval $[\frac{\pi}{2},\beta_{r_0}]$.

The computation in Sage gives the value
\[\delta_N=\delta_{1000}=0.6194\ldots\]
which implies that if $\frac{m}{\ell}<\delta_N$ for $k\gg 0$,
$g_{k,m}$ has at least
\[\left\lfloor \frac{k}{3}- m\right\rfloor-\left\lceil
\frac{k}{4}\right\rceil=\ell-m\]
roots on $\mathcal{A}$. Since $g_{k,m}$ can have at most
$\ell-m$ roots in the upper half plane, we conclude that every
root of $g_{k,m}$ lies on $\mathcal{A}$ and define
$\mathcal{P}(\delta)=1$ for $\delta<0.6194$.

Suppose that $\delta_{r_0+1}\leq \delta<\delta_{r_0}$. Then
$g_{k,m}$ has at least
\[\left\lfloor \frac{k \beta_{r_0}}{2\pi}+2 m\cos
\beta_{r_0}\right\rfloor-\left\lceil
\frac{k}{4}\right\rceil>\frac{k \beta_{r_0}}{2\pi}+2 m\cos
\beta_{r_0}- \frac{k}{4}-2>\frac{6\ell \beta_{r_0}}{\pi}+2 m\cos
\beta_{r_0}- 3\ell-2\]
roots on $\mathcal{A}$. 
This implies that
\[\frac{1}{D}\left|\left\{\textrm{zeros of }g_{k,m}\textrm{ on
}\mathcal{A}\right\}\right|\geq
\frac{\frac{6}{\pi}\beta_{r_0}+2 \delta \cos
  \beta_{r_0}-3}{1- \delta}-\frac{2}{\ell}.\]
Fixing a small $\varepsilon$, say $\varepsilon = 10^{-10}$, if
$k>24/\varepsilon$ we may define
\[\mathcal{P}(\delta)=\min(0,\frac{\frac{6}{\pi}\beta_{r_0}+2 \delta \cos
  \beta_{r_0}-3}{1- \delta}-\varepsilon)\]
for all $\delta\in (\delta_{r_0+1},\delta_{r_0}]$.

Finally, set $\mathcal{P}(\delta)=0$ for $\delta\in
[\frac{3}{\pi},1]$. We notice that $\mathcal{P}(\delta)=0$ for $\delta\geq 0.9546$. 

\begin{figure}[H]
\centering
\includegraphics[scale=0.15]{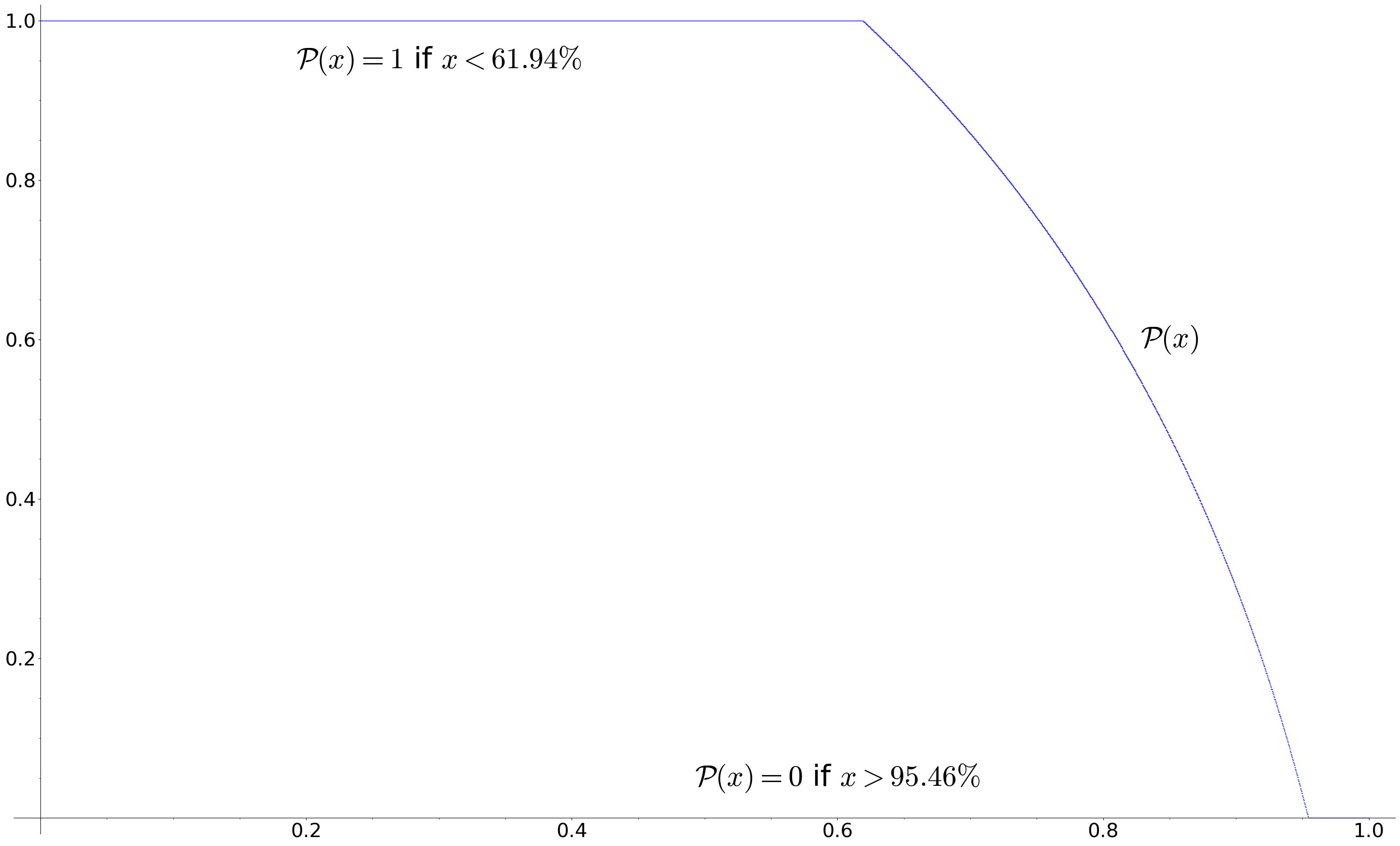}
\caption{Lower bound $\mathcal{P}$ for the proportion of roots on $\mathcal{A}$}
\end{figure}


The first lower bound comes from the fact that the linear function connecting $(0.6194,1)$ and $(0.9546,0)$ lies below the graph of $\mathcal{P}$.

We may similarly define a function $\mathcal{T}(\delta)$ by
setting $\mathcal{T}(\delta)=\frac{2\pi}{3}$ for
$\delta<0.6194$, $\mathcal{T}(\delta)=\beta_{r_0}$ if
$\delta_{r_0+1}\leq \delta<\delta_{r_0},\frac{3}{\pi}$, and
$\mathcal{T}(\delta)=\frac{\pi}{2}$ otherwise.
\begin{figure}[H]
\centering
\includegraphics[scale=0.15]{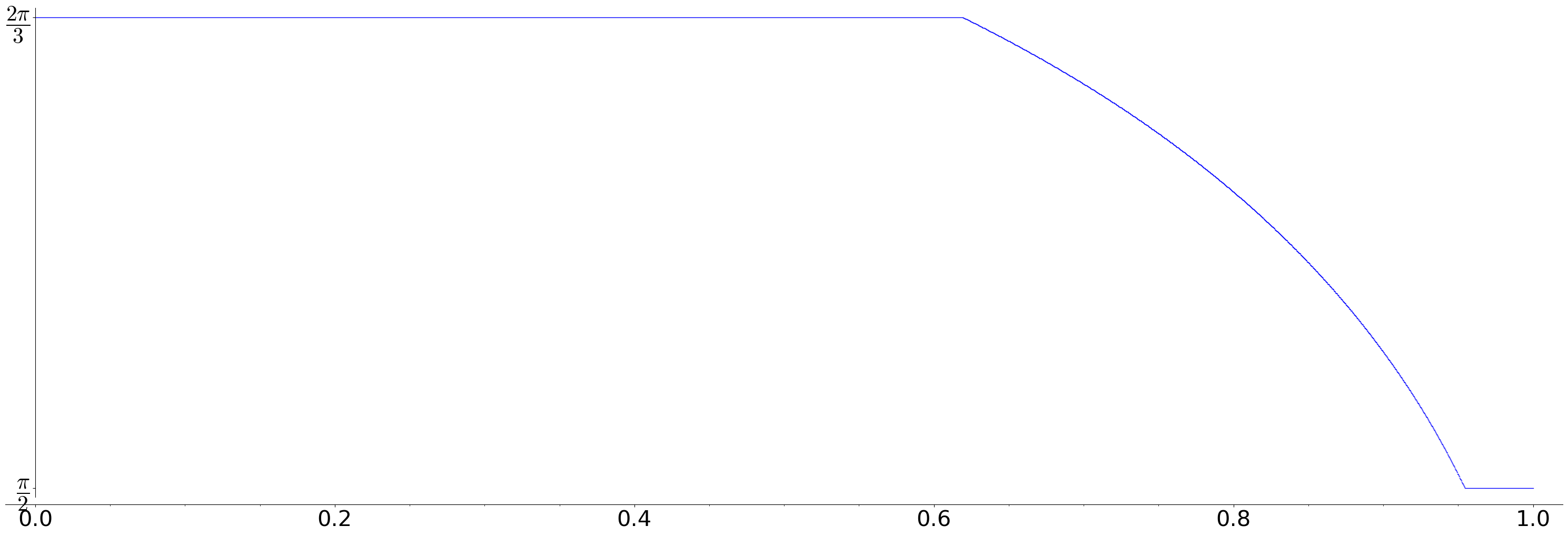}
\caption{Angle $\mathcal{T}(\delta)$ with $\mathcal{A}_{\frac{\pi}{2},\mathcal{T}(\delta)}$ containing Miller zeros\label{fig:calT1000}}
\end{figure}
By the above, $\overline{g}_{k,m}(e^{i\theta})=2\cos(k\theta/2+2\pi
m\cos\theta)+o(1)$ on the arc $[\frac{\pi}{2},\mathcal{T}(\delta)]$ and the equidistribution was proved in \cite[\S3.4]{raveh:miller} on any arc where this asymptotics holds. 

The conclusion follows form the fact that $\Theta$ is the piece-wise linear lower hull of $\mathcal{T}$.
\end{proof}

\section{Zeros of weakly holomorphic forms}\label{Weakly-holom-Miller}

In this section, we extend our previous result to weakly holomorphic modular forms, i.e., 
modular forms that are holomorphic on $\h$ with poles allowed at the cusp $i\infty$. Writing again $k=12\ell+k'$ for some $\ell\in \Z$ and $k'\in\{0,4,6,8,10,14\}$, the space of weight $k$ weakly holomorphic modular forms for $\SL_2(\mathbb{Z})$ is infinite-dimensional with a basis consisting of forms $g_{k,m}=q^{m}+O(q^{\ell+1})$ for $m\leq \ell$.




If $m\le \ell-|\ell|$, Duke and Jenkins \cite{duke-jenkins} showed that all roots of
$g_{k,m}$ are on the arc $\mathcal{A}$. We refine their result by showing that, for $k\ll 0$, all finite roots of $g_{k,m}$ are
on $\mathcal{A}$ whenever $\dfrac{m}{\ell}>115.98\%$, and moreover $g_{k,m}$
has some roots on $\mathcal{A}$ if $\dfrac{m}{\ell}>110.26\%$. This
is close to the experimentally obtained bounds, as seen in the following table. 
\begin{center}
\begin{tabular}{llll}
$k$ & $\ell$ & Not all roots on $\mathcal{A}$ & No roots on
                                 $\mathcal{A}$\\
  \hline
  $-12000$ & $-1000$ & $m\geq -1154$ & $m\geq -1090$ \\
  $-12000+4$ & $-1000$ & $m\geq -1155$ & $m\geq -1096$ \\
  $-12000+6$ & $-1000$ & $m\geq -1151$ & $m\geq -1089$ \\
  $-12000+8$ & $-1000$ & $m\geq -1156$ & $m\geq -1105$ \\
  $-12000+10$ & $-1000$ & $m\geq -1152$ & $m\geq -1096$ \\
  $-12000+14$ & $-1000$ & $m\geq -1153$ & $m\geq -1104$
\end{tabular}
\end{center}

Given the similarity to the proof of Theorem~\ref{t:main-arc}, we outline the argument and concentrate on the adjustments needed in the present context. 

\begin{lemma}\label{l:Deltamax}
If $B\geq 0.287$, the maximum of $|\Delta(x+iB)|$ is
  attained when $x=\frac{1}{2}$.
\end{lemma}
\begin{proof}
Consider the function
\[f(x)=\log |\Delta(x+iB)|^2=-2\pi B+\sum_{n=1}^\infty (1-2e^{-2\pi B n}\cos (nx)+e^{-4\pi B n})\]
on
$[-\frac{1}{2},\frac{1}{2}]$, with derivative $\displaystyle
f'(x)=2\sum_{n=1}^\infty ne^{-2\pi B n}\sin (nx)$. It suffices
to show that $f'(x)\geq 0$ on $[0,\frac{1}{2}]$ by symmetry. In
fact, we'll show that
\[\frac{f'(x)}{2\sin x}=e^{-2\pi B}+\sum_{n=2}^\infty ne^{-2\pi
  nB}\frac{\sin (nx)}{\sin x}\geq 0\]
on this interval. Writing $A=e^{-2\pi B}<0.1648$ we want to show
that $\displaystyle A>\sum_{n=2}^\infty nA^n\frac{\sin (nx)}{\sin x}$.
But Chebyshev polynomials imply that $|\frac{\sin (nx)}{\sin
  x}|\leq n$ and whenever $A<0.1648$ we have
\[A>\frac{A^2+A}{(1-A)^3}-A=\sum_{n=2}^\infty n^2A^n\geq \sum_{n=2}^\infty nA^n|\frac{\sin (nx)}{\sin x}|.\]
\end{proof}
\begin{remark}
With a little more care, one can show that Lemma \ref{l:Deltamax} holds for $B\geq 0.25$. However, it fails, for instance, when $B=0.235$.
\end{remark}

\begin{lemma}\label{l:IB-weak}
Let $\frac{\pi}{2}\leq \alpha\leq \beta \leq\frac{2\pi}{3}$. 
For each $B>0.287$ such that $\Im \gamma e^{i\theta}\neq B$ for any $\gamma$ as above and all $\theta\in [\alpha,\beta]$, there is a constant $h_{\alpha,\beta,B}$ such that
\[\int_{-\frac{1}{2}+iB}^{\frac{1}{2}+iB}|G(\tau,e^{i\theta})|d\tau\leq e^{-2\pi m B+h_{\alpha,\beta,B}}\left(\frac{|\Delta(e^{i\theta})|}{|\Delta(\frac{1}{2}+iB)|}\right)^\ell .\]
\end{lemma}
\begin{proof}
The proof is identical to Lemma \ref{l:IB}, except in using Lemma \ref{l:Deltamax} to show that if $\ell<0$ \[\displaystyle \max_{|x|\leq \frac{1}{2}}\left(\frac{|\Delta(e^{i\theta})|}{|\Delta(x+iB)|}\right)^\ell = \left(\frac{|\Delta(e^{i\theta})|}{|\Delta(\frac{1}{2}+iB)|}\right)^\ell.\] 
\end{proof}

\begin{lemma}\label{l:miller-approx-arc-weak}
Let $\alpha, \beta$ as above and $0.287<B<\sin \alpha$, and let $k<0$. 
There exists a constant $I^{\w}(\alpha,\beta,B)\geq 1$ such that 
for all $\theta\in [\alpha,\beta]$ we have
\[\overline{g}_{k,m}(e^{i \theta})=2\cos\left(2\pi m\cos
(\theta)+k\theta/2\right)+o(1),\]
as $k\to -\infty$ and $\frac{m}{\ell}>I^w(\alpha,\beta,B)$.
\end{lemma}
\begin{proof}
As in Lemma \ref{l:miller-approx-arc-weak}, write $\delta=\frac{m}{\ell}\in [1,\infty)$ and denote $\mathcal{R}$ the finite set of elements $\gamma\in \SL_2(\mathbb{Z})$ such that $\Im \gamma e^{i\theta}\geq B$ for some $\theta\in [\alpha,\beta]$, excluding $I_2$ and $S$. From Lemma \ref{l:IB-weak} we conclude that 
\[\overline{g}_{k,m}(e^{i \theta})-2\cos\left(2\pi m\cos
(\theta)+k\theta/2\right)=o(1)\]
as $k\to-\infty$ as long as 
\begin{align}\label{eq:upperboundweak}
2\pi \delta(\sin\theta-B) +\ln
  |\Delta(e^{i\theta})|-\ln|\Delta(\frac{1}{2}+iB)|&>0\\
  2\pi\delta \sin\theta(1-|ce^{i\theta}+d|^{-2})-12\ln|ce^{i\theta}+d|&>0\ \ \forall \gamma\in \mathcal{R}.\notag
\end{align}
The bounds \eqref{eq:upperboundweak} are guaranteed for
all such $\theta\in [\alpha,\beta]$ if
\begin{align}\label{eq:upperboundabweak}
2\pi \delta (\sin\beta-B) -
  (\ln|\Delta(\frac{1}{2}+iB)|-\ln|\Delta(e^{i\alpha})|)&>0\\
  2\pi
  \delta\sin\beta(1-|ce^{i\beta}+d|^{-2})-12\ln|ce^{i\alpha}+d|&>0\ \ \ \forall \gamma\in \mathcal{R}\ \ cd\geq 0\notag\\
  2\pi
  \delta \sin\beta(1-|ce^{i\alpha}+d|^{-2})-12\ln|ce^{i\beta}+d|&>0\ \ \ \forall \gamma\in \mathcal{R}\ \ cd<0.\notag
\end{align}
Finally, consider $I^{\w}(\alpha, \beta, B)$ the
maximum of the following:
\begin{itemize}
\item $\dfrac{\ln|\Delta(\frac{1}{2}+iB)|-\ln|\Delta(e^{i\alpha})|}{2\pi
  (\sin\beta-B)}$
\item $\dfrac{12\ln|ce^{i\alpha}+d|}{2\pi\sin \beta
  (1-|ce^{i\beta}+d|^{-2})}$ for all $cd\geq 0$ and
\item $\dfrac{12\ln|ce^{i\beta}+d|}{2\pi\sin \beta (1-|ce^{i\alpha}+d|^{-2})}$ for all $cd< 0$.
\end{itemize}
\end{proof}

\begin{theorem}\label{t:main-arc-weak}
There exists a piece-wise constant function $\mathcal{P}_-(x)$, whose graph appears below, such that for $k\ll 0$ the proportion of roots of
$g_{k,m}$ on the boundary arc is at least
$\mathcal{P}_-(\frac{m}{\ell})$.

Moreover, there exists a piece-wise linear function $\mathcal{T}_-:[0,1]\to [\frac{\pi}{2},\frac{2\pi}{3}]$, such that for $\delta>1$, roots of $g_{k,m}$ equidistribute on the arc $[\mathcal{T}_-(\delta),\frac{2\pi}{3}]$ as $\frac{m}{\ell}>\delta$. 
\end{theorem}
\begin{proof}
The proof is identical to that of Theorem \ref{t:main-arc}, using an explicit search in Sage.
\end{proof}


\begin{figure}[H]
\centering
\includegraphics[scale=0.15]{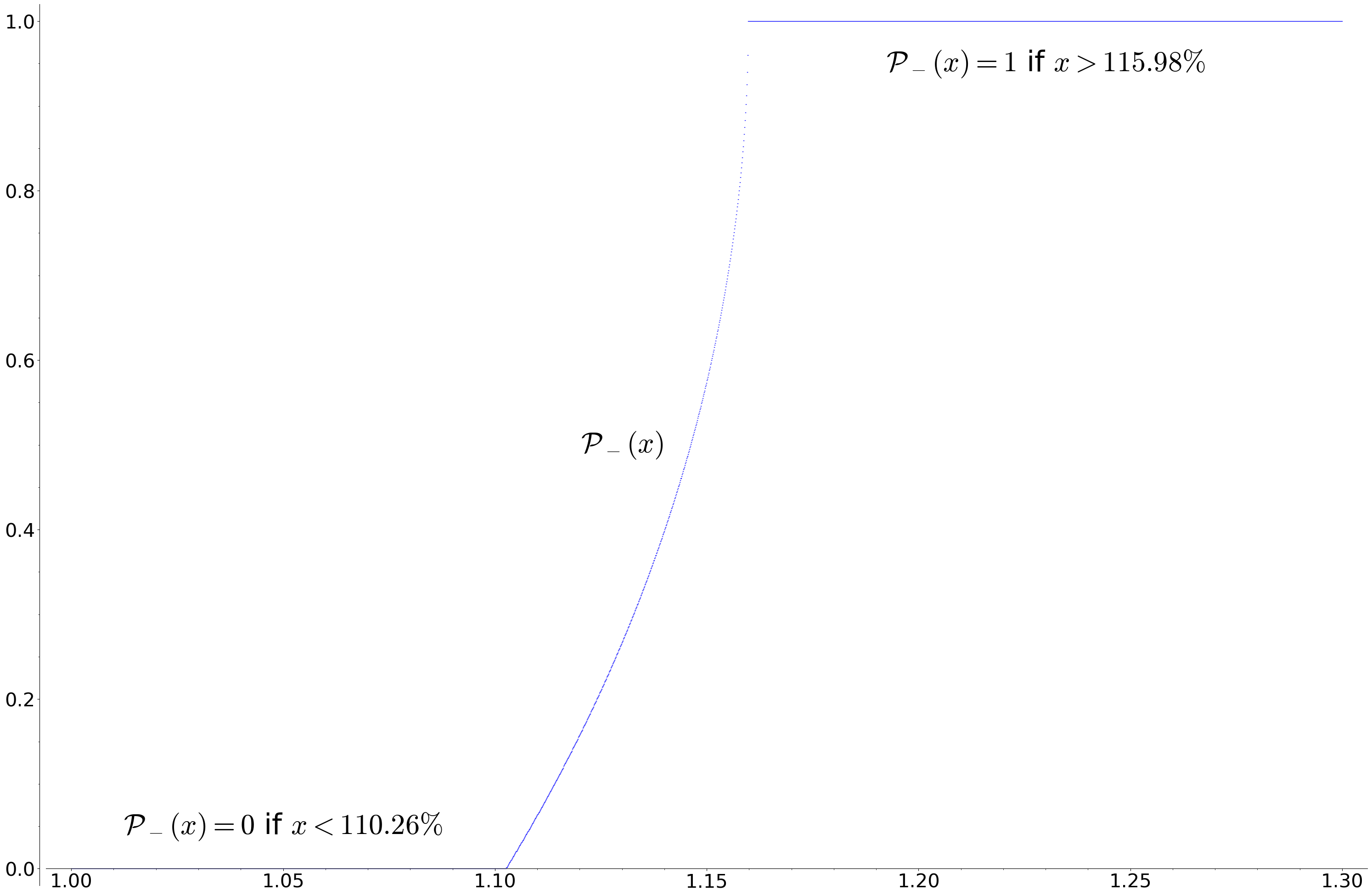}
\caption{Lower bound $\mathcal{P}_-$ for the proportion of roots on $\mathcal{A}$}
\end{figure}

\begin{figure}[H]
\centering
\includegraphics[scale=0.15]{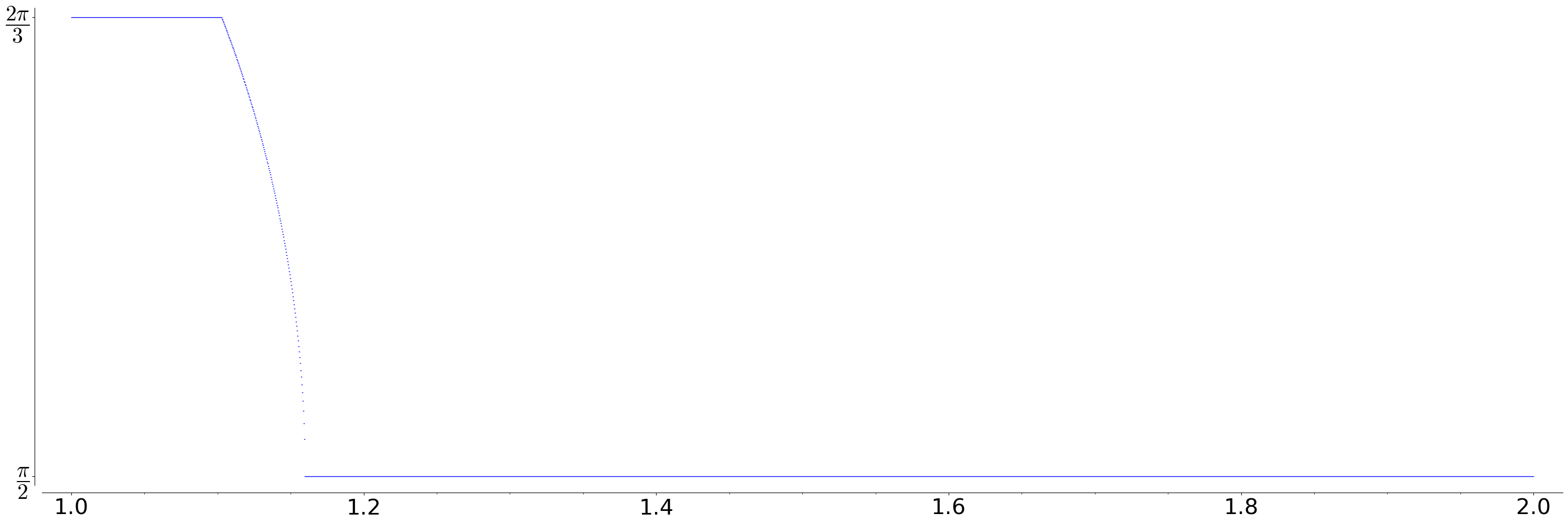}
\caption{Angle $\mathcal{T}_-(\delta)$ with $\mathcal{A}_{\frac{\pi}{2},\mathcal{T}_-(\delta)}$ containing Miller zeros}
\end{figure}

\section{Proof of Theorem B}

Consider again the Miller basis $g_{k,m}=q^m+O(q^{\ell+1})$ of weight $k=12\ell+k'$, with $k'\in \{0,4,6,8,10,14\}$. For some monic polynomial $P$ of degree $D=\ell-m$ we have
\[\frac{g_{k,m}}{\Delta^\ell E_{k'}}=P(j),\] where $j=E_4^3/\Delta$ is the usual $j$-invariant.  
We refer to $P$ as the Faber polynomial of $g_{k,m}$. Writing
\[ P(X)=\sum_{d=0}^D y_dx^{D-d} \quad  y_0=1,\] we will be concerned with estimating the coefficients $y_d$ when $D$ grows with $k$, albeit more slowly.  
For all $r\ge 1$ we also set \[ j^r=\sum_{n=-r}^\infty c_{r,n}q^n, \]  so that
\[P(j)=\sum_{d=0}^D y_{d}j^{D-d}=\sum_{d=0}^D\sum_{n\geq
  -(D-d)}y_dc_{D-d,n}q^{n}.\]
Since
\[\frac{g_{k,m}}{\Delta^\ell E_{k'}}=q^{-D}\frac{1+O(q^{D+1})}{(q^{-1}\Delta)^\ell E_{k'}},\]
it follows that
\[(q^{-1}\Delta)^{-\ell}(E_{k'})^{-1}=\sum_{d=0}^D\sum_{n=
  -(D-d)}^0y_dc_{D-d,n}q^{D+n}+O(q^{D+1}).\]
On the other hand, we may expand the power series formally as
\[(q^{-1}\Delta)^{-\ell}(E_{k'})^{-1}=\sum_{d=0}^\infty\mathcal{D}_{\ell,d}q^d.\]
\begin{lemma}\label{l:rudnick-coeffs}
  \begin{enumerate}
\item If $d=O(|k|^{\frac{1}{2}-\varepsilon})$ for some $\varepsilon\in (0,\frac{1}{2})$ then
\[\mathcal{D}'_{\ell,d}, \mathcal{D}_{\ell,d}=\frac{(2|k|)^d}{d!}\left(1+O\left(|k|^{-2\varepsilon}\right)\right).\]
\item If $d\leq |\ell|$ then
\[\mathcal{D}'_{\ell,d}, \mathcal{D}_{\ell,d}=O\left(\left(\frac{2 \alpha
  |k|}{d}\right)^d\right),\]
where $\alpha=\left(\frac{25}{24}\right)^{25}$.
\end{enumerate}
\end{lemma}
\begin{proof}
Suppose $k>0$, as the proof when $k<0$ is almost identical.
As in \cite[\S4]{rudnick:faber}, we have 
\[\mathcal{D}'_{\ell,d}=\sum_{\sum i
  r_i=d}\prod_i\binom{2k+r_i-1}{r_i}.\]
By the Stirling approximation, we have
\begin{align*}
  \binom{2k+r-1}{r}&=\frac{2k}{2k+r}\binom{2k+r}{r}\\
  &=\frac{2k}{2k+r}\frac{(2k+r)^{2k+r}}{r!(2k)^{2k}}\frac{1}{e^r}\sqrt{\frac{2k+r}{2k}}\left(1+\frac{1}{24k}+O(k^{-2})\right)\\
  &=\frac{(2k)^r}{r!}\left[\frac{1}{e}\left(1+\frac{r}{2k}\right)^{1+\frac{2k}{r}-\frac{1}{2r}}\right]^r\left(1+\frac{1}{24k}+O(k^{-2})\right).
\end{align*}

(1) If $r\leq d=O(k^{\frac{1}{2}-\varepsilon})$, we see that
\begin{align*}
\log
  \left[\frac{1}{e}\left(1+\frac{r}{2k}\right)^{1+\frac{2k}{r}-\frac{1}{2r}}\right]^r&=(r+2k-\frac{1}{2})\log
                                                                            \left(1+\frac{r}{2k}\right)-r\\
  &=(r+2k-\frac{1}{2})\left(\frac{r}{2k}-\frac{r^2}{8k^2}+O\left(\frac{r^3}{k^3}\right)\right)-r\\
  &=\frac{r^2}{4k}-\frac{r}{4k}+\frac{r^2}{16k^2}+O\left(\frac{r^3}{k^2}\right)\\
 &=\frac{r^2}{4k}+O\left(\frac{r}{k}\right)=O(k^{-2 \varepsilon}).
\end{align*}
We conclude that
\begin{align*}
\sum_{\sum i
  r_i=d}\prod_i\binom{2k+r_i-1}{r_i}&=\sum_{\sum i
  r_i=d}\prod_i\frac{(2k)^{r_i}}{r_i!} \left(1+\frac{r_i^2}{4k}+O\left(\frac{r}{k}\right)\right)\left(1+\frac{1}{24k}+O(k^{-2})\right)^{d'},
\end{align*}
where $d'<\sqrt{2d}$ is the number of nonzero $r_i$. Thus
$\left(1+\frac{1}{24k}+O(k^{-2})\right)^{d'}=1+O(k^{-\frac{3}{4}-\frac{\varepsilon}{2}})$,
while $\prod
\left(1+\frac{r_i^2}{4k}+O\left(\frac{r}{k}\right)\right)=1+O(k^{-2
  \varepsilon})$ for any choice of $r=(r_1,\ldots,
r_d)$. Therefore, it suffices to verify that the term with
$r=(d,0,\ldots,0)$ dominates the sum.

Note that the sum
\begin{align*}
\sum_{\sum i
  r_i=d}\prod_i\frac{(2k)^{r_i}}{r_i!}&=\sum_{j=1}^d\binom{d-1}{j-1}\frac{(2k)^j}{j!}
\end{align*}
can be computed using the Laguerre polynomial
$L_d^{-1}(-2k)$. The terms in the latter sum are increasing in
$j$ as $d=O(k^{\frac{1}{2}-\varepsilon})$ so
\begin{align*}
\sum_{\sum i
  r_i=d}\prod_i\frac{(2k)^{r_i}}{r_i!}&=\frac{(2k)^d}{d!}+\frac{(2k)^{d-1}}{(d-1)!}+O\left(\frac{(d-2)(2k)^{d-2}}{(d-2)!}\right)\\
  &=\frac{(2k)^d}{d!}\left(1+\frac{d}{2k}+O\left(\frac{d^3}{k^2}\right)\right)
\end{align*}
and the desired formula follows.

(2) If $r\leq d\leq \ell$, we have
\[\left[\frac{1}{e}\left(1+\frac{r}{2k}\right)^{1+\frac{2k}{r}-\frac{1}{2r}}\right]^r<\left[\frac{1}{e}\left(1+\frac{r}{2k}\right)^{1+\frac{2k}{r}}\right]^r<e^{-r}\alpha^r,\]
as $\frac{r}{2k}\leq \frac{\ell}{2k}=\frac{1}{24}$. As above,
the term that dominates corresponds to $r_1=d$, in which case
Stirling's approximation gives
\[\mathcal{D}'_{\ell,d}=O\left(\frac{(2 e^{-1}\alpha k)^d}{d!}\right)=O\left(\left(\frac{2 \alpha k}{d}\right)^d\right).\]

Turning our attention to the coefficients $\mathcal{D}_{\ell,d}$, we recall that the coefficients in the expansion
\[\frac{1}{E_{k'}}=\sum_{n\ge 0} \beta_n q^n,\]
satisfy $\beta_n = O(e^{2\pi n})$. Indeed, when $k'=6,10,14$ this follows from \cite[Theorem 1]{heim-neuhauser}; for $k'=4$ one has 
 $\beta_n=O(e^{\pi\sqrt{3}n})$ by \cite[Theorem 3]{heim-neuhauser}), and the statement for $k'=8$ is clear since $E_8=E_4^2$. Thus, 
\[\mathcal{D}_{\ell,d}=\sum_{i=0}^d
\mathcal{D}'_{\ell,i}\beta_{d-i}=\frac{(2k)^d}{d!}\left(1+O\left(\frac{1}{k^{\frac{1}{2}+\varepsilon}}\right)\right),\]
if $d=O(k^{\frac{1}{2}-\varepsilon})$, while if $d\leq \ell$, we have $\mathcal{D}_{\ell,d}=O\left(\left(\frac{2 \alpha
  k}{d}\right)^d\right)$.

\end{proof}

Finally, we recall some estimates for the coefficients of the powers of $j$. Namely, according to \cite[Proposition 4.1]{brisebarre-philibert}, for all $r\ge 1$ and $n\in \Z$ 
such that $n\ge -r$ we have
\[ c_{r,n} \leq e^{2\pi n}1728^r.\] Sharper estimates are available when $-r+1\le n\le -rc$, where $c=\dfrac{e^{2\pi}}{1728}$. In this range, \cite[Proposition 4.2]{brisebarre-philibert} 
says that
\[ c_{r,n}\le (1728-e^{2\pi}) ^{r+n} \frac{ (-n)^n r^r} { (r+n)^{r+n} }.\] 


Having assembled the necessary ingredients, we proceed to the central technical component of this section.

\begin{proposition}\label{p:faber-coefficients}
If $D=O(|k|^{1/2-\varepsilon})$ for some $\varepsilon>0$, we have
\[y_d=\frac{(2|k|)^d}{d!}\left(1+O(|k|^{-2 \varepsilon})\right),\]
for each $0\leq d\leq D$.
\end{proposition}
\begin{proof}
Again, we assume $k>0$, the other case being similar.
We use induction on $d$; the base case is immediate because $y_0=1$. 

Suppose $y_r=\frac{(2k)^r}{r!}\left(1+O(k^{-2
  \varepsilon})\right)$ holds for all $0\leq r\leq d-1$. Since
\begin{equation}\label{eq:faber}\sum_{r=0}^d y_r
c_{D-r,-(D-d)}=\mathcal{D}_{\ell,d}=\frac{(2k)^d}{d!}(1+O(1/k^{\frac{1}{2}+\varepsilon})),
\end{equation}
it suffices to check that 
\[\frac{y_d-\mathcal{D}_{\ell,d}}{(2k)^d/d!}=\frac{d!}{(2k)^d}\sum_{r=0}^{d-1}y_rc_{D-r,-(D-d)}=O\left(\sum_{r=0}^{d-1}\frac{d!}{(2k)^{d-r}r!}c_{D-r,-(D-d)}\right)\]
is, in fact, $O(k^{-2 \varepsilon})$.

When $D-d>c(D-r)$, and hence whenever $d<D(1-c)$, we know that 
\[c_{D-r,-(D-d)}\leq
(1728-e^{2\pi})^{d-r}\frac{(D-r)^{D-r}}{(D-d)^{D-d}(d-r)^{d-r}}.\]
Otherwise, the bound $c_{D-r,-(D-d)}\leq 1728^{D-r}e^{-2\pi(D-d)}$ still applies. 

Thus, for $d<D(1-c)$, we have
\begin{align*}
\sum_{r=0}^{d-1}\frac{d!}{(2k)^{d-r}r!}c_{D-r,-(D-d)}&\leq
                                                       \sum_{r=0}^{d-1}\frac{d!}{(2k)^{d-r}r!}(1728-e^{2\pi})^{d-r}\frac{(D-r)^{D-r}}{(D-d)^{D-d}(d-r)^{d-r}}
\end{align*}
When $r=0$, the term above is
\begin{align*}
\frac{d!}{(2k)^d}\frac{(1728-e^{2\pi})^dD^D}{(D-d)^{D-d}d^d} &=O
\left(\sqrt{d}\left(\frac{d(1728-e^{2\pi})}{ke}\right)^d
\left(1+\frac{d}{D-d}\right)^{D-d}\right) \notag \\
&=O\left(\left(\frac{dC}{k}\right)^d\right)=O(k^{-2 \varepsilon})
\end{align*}
for any $C>1728-e^{2\pi}$.

If $r\geq 1$, Stirling's bound gives
\begin{align*}
& \frac{d!}{(2k)^{d-r}r!}(1728-e^{2\pi})^{d-r}\frac{(D-r)^{D-r}}{(D-d)^{D-d}(d-r)^{d-r}} \notag \\
&=O\left(\frac{d^d}{k^{d-r}r^r}(1728-e^{2\pi})^{d-r}\frac{(D-r)^{D-r}}{(D-d)^{D-d}(d-r)^{d-r}}\right),
\end{align*}
as $\sqrt{d/r}\leq 2^{d-r}$.
Note that the derivative with respect to $r$ of log of the big‑O term is
\[ \log(k)-\log(1728-e^{2\pi})-\log(D-r)+\log(d-r),\] which is always
positive if $D<k/(1728-e^{2\pi})$.

Fix $t>0$ such that $\left(\frac{1}{2}-\varepsilon \right)(2t+1)<t-2 \varepsilon$. Then $\frac{|y_d-\mathcal{D}_{\ell,d}|}{(2k)^d/d!}$ is
\begin{align*}
& O\left(\frac{(d-t)d^d}{k^{t}(d-t)^{d-t}}(1728-e^{2\pi})^{t}\frac{(D-d+t)^{D-d+t}}{(D-d)^{D-d}t^t}+\frac{td^d(1728-e^{2\pi})}{k(d-1)^{d-1}}\frac{(D-(d-1))^{D-(d-1)}}{(D-d)^{D-d}}\right)\\
  &=O\left( \frac{d(d(D-d+t))^tC_1^t}{k^tt^t} + \frac{td(D-d+t)C_2}{k}\right)=O\left( \frac{D^{2t+1}C_1^t}{k^t} + \frac{tD^2C_2}{k}\right)=O(k^{-2 \varepsilon}),
\end{align*}
for some constants $C_1$ and $C_2$.

If $d>D(1-c)$, we have to break up the sum into pieces based on $D-r$. 
\begin{align*}
  \frac{|y_d-\mathcal{D}_{\ell,r}|}{(2k)^d/d!}&=O\left(
                        \sum_{r<D-c^{-1}(D-d)}\frac{d!}{(2k)^{d-r}r!}1728^{D-r}e^{-2\pi(D-d)}\right)\\
  &\
    +O\left(\sum_{r>D-c^{-1}(D-d)}^{d-1}\frac{d!}{(2k)^{d-r}r!}(1728-e^{2\pi})^{d-r}\frac{(D-r)^{D-r}}{(D-d)^{D-d}(d-r)^{d-r}}\right)\\
  &=O\left(
                        \sum_{r<D-c^{-1}(D-d)}\frac{d^d}{k^{d-r}r^r}1728^{D-r}e^{-2\pi(D-d)}\right)+O(k^{-2
    \varepsilon})
\end{align*}
As in the previous situation, we see that the term in the sum
increases with $r$ as long as $D<k/1728$.

When $d\leq D-1$, in the above sum we take $r=D-c^{-1}(D-d)$ and the sum is
\[=O\left(\frac{Dd^d}{(D-c^{-1}(D-d))^{D-c^{-1}(D-d)}}\left(\frac{1728^{1/c}}{e^{2\pi}k^{\frac{1-c}{c}}}\right)^{D-d}\right).\]
The derivative with respect to $d$ of log of the above term is
\[\log d-c^{-1}\log (D-c^{-1}(D-d))+\frac{1-c}{c}\log k-O(1),\]
which is positive whenever $D=o(k)$. Therefore, we may take
$d=D-1$ and the sum is
\[=O\left(\frac{D(D-1)^{D-1}}{(D-c^{-1})^{D-c^{-1}}}\left(\frac{1728^{1/c}}{e^{2\pi}k^{\frac{1-c}{c}}}\right)\right)=O(k^{-1/2}),\]
as $D=O(k^{1/2-\varepsilon})$.

Finally, when $d=D$, $r\leq D-1$ and the terms in the sum are
maximized when $r=D-1$. Therefore, the sum is
\[O\left(\frac{(D-1)D^D}{k(D-1)^{D-1}}1728e^{-2\pi}\right)=O\left(\frac{D^2}{k}\right)=O(k^{-2 \varepsilon}).\]
\end{proof}

We now apply the above approximation to Faber polynomials to the
study of roots of $g_{k,m}$.
\begin{corollary}
Suppose $D=O(k^{\frac{1}{2}-\varepsilon})$ for some $\varepsilon>0$. Then $g_{k,m}$ has no non-elliptic roots on $\mathcal{A}$ for $k\gg 0$. 
\end{corollary}
\begin{proof}
By Proposition \ref{p:faber-coefficients}, the Faber polynomial
of $g_{k,m}$ is
\[\sum_{r=0}^D \frac{(2k)^r}{r!}\left(1+O(k^{-2
  \varepsilon})\right)x^{D-r}.\]
In particular, for $k\gg 0$, all the coefficients are positive and the polynomial cannot have nonnegative roots. Since $j(\mathcal{A})=[0,1728]$, the form $g_{k,m}$ cannot have non-elliptic roots on $\mathcal{A}$. 
\end{proof}

Finally, when $D$ is much smaller than $k$, we can prove Theorem~\ref{tl:conj-tinyD} on the location of the zeros of $g_{k,m}$. To this end, we will use the following stability result for the preimage of the Szeg\H{o} curve under the $j$-invariant.

\begin{proposition}\label{p:j-szego}
For a real number $\delta>0$, let $\mathcal{S}_\delta$ be the curve of points $\tau\in \mathcal{F}$ such that $\frac{24}{(1-\delta)j(\tau)}\in \mathcal{S}$. Then $\mathcal{S}_\delta-\frac{1}{2\pi}\log|1-\delta|$ approaches $\mathcal{L}_\pm$ as $\delta\to 1^\mp$. 
\end{proposition}
\begin{proof}
Note that $|j(z)|\gg 0$ then $\Im z=\frac{\log |j(z)|}{2\pi}+o(1)$ and the conclusion follows.
\end{proof}

\begin{theorem}\label{t:conj-tinyD}
Suppose $D\to\infty$ such that $D<\frac{\alpha \log |k|}{\log\log |k|}$ for some $\alpha\in (0,1)$.
Then the zeros of $g_{k,m}$ asymptotically approach
$\mathcal{S}_\delta=\mathcal{L}_\pm -\frac{1}{2\pi}\log|1-\delta|$,
where $\delta=\frac{m}{\ell}$ and $\pm$ is the sign of $k$.

\begin{figure}[H]
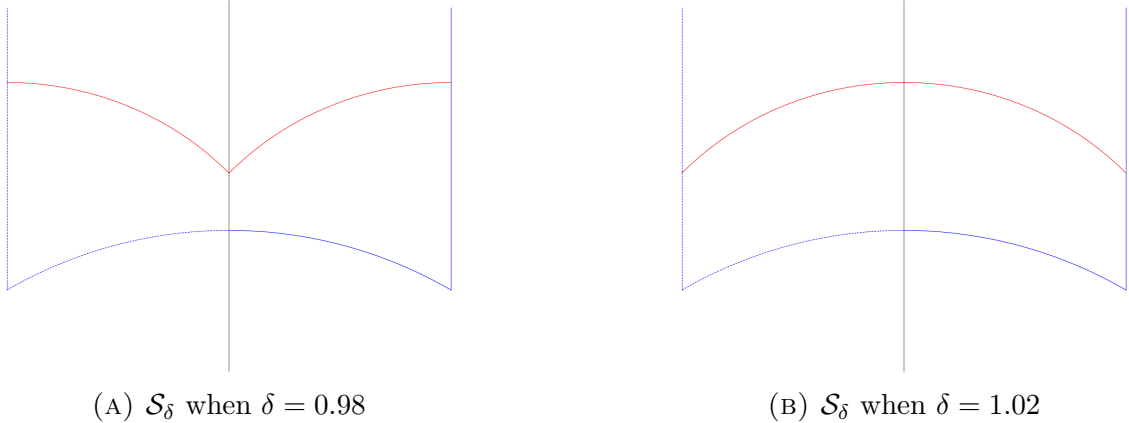

     \centering
     \begin{subfigure}[b]{0.45\textwidth}
         \centering
         \includegraphics[width=\textwidth]{inverse-szego.png}
         \subcaption{$\mathcal{S}_\delta$ when $\delta=0.98$}
         \label{fig:positive_delta}
     \end{subfigure}
     \hfill
     \begin{subfigure}[b]{0.45\textwidth}
         \centering
         \includegraphics[width=\textwidth]{inverse-szego-negative.png}
         \subcaption{$\mathcal{S}_\delta$ when $\delta=1.02$}
         \label{fig:negative_delta}
     \end{subfigure}
     \label{fig:szego_comparison}
     \caption{Asymptotic location of zeros as $\delta\to 1$}
\end{figure}

\end{theorem}
\begin{proof}
Let $E_D(x)=\sum\limits_{k=0}^Da_kx^{D-k}$, where $a_k=\frac{1}{k!}$, in which
case, by Proposition \ref{p:faber-coefficients} applied to $\varepsilon\in (\alpha/2,1/2)$, the Faber polynomial satisfies
\[\frac{F(2kx)}{(2k)^D}=\sum_{k=0}^Db_kx^{D-k},\]
where $b_k=a_k(1+O(k^{-\alpha}))$.
As in \cite[\S 5.1]{rudnick:faber}, we use Ostrowski's theorem
which states that the roots of $E_D(x)$ and $F(2kx)$ are within
$2D \Gamma(\sum |a_k-b_k|\Gamma^{-k})^{1/D}$, where $\Gamma=\max
(a_k^{1/k},b_k^{1/k})$.

In our case, we may take $\Gamma=2$ for $k\gg 0$, in which case
\[\sum |a_k-b_k|\Gamma^{-k}=\sum_{k=0}^D \frac{O(k^{-\alpha})}{k! \Gamma^k}=O(e^{-\Gamma} k^{-\alpha}).\]
Thus, the roots of $F(2k x)$ and $E_D(x)$ are within $O(D
e^{-\Gamma/D}k^{-\alpha/D})=O(Dk^{-\alpha/D})$ of each
other. The hypothesis implies that $D=\frac{f(k)\log k}{\log\log k}$ for some $f(k)<\alpha$ so 
\[\log D-\frac{\alpha \log k}{D}=(1-\frac{\alpha}{f(k)})\log\log k +\log f(k)-\log\log\log k\to -\infty\]
so the error above is $o(1)$.

Finally, by \cite{szego} the roots of $E_D(\frac{1}{Dx})$ asymptotically approach the curve $\mathcal{S}$. However, by the above, the roots of $E_D(\frac{1}{Dx})$ asymptotically approach the roots of $F(\frac{2k}{D x})=F(\frac{24}{(1-\delta)x})$, and the result follows from Proposition \ref{p:j-szego}. 
\end{proof}

\section{Conjectural behavior of zeros of Miller forms} \label{Conject}
The asymptotic between the zeros of the Miller forms $g_{k,m}$
and the logarithmic Szeg\H{o} curve holds experimentally for all
values of $\delta=\frac{m}{\ell}$.

\begin{conjecture}\label{conj}
For each $\delta>0$, the zeros of $g_{k,m}$ asymptotically
approach the curve $\mathcal{C}_\delta$, defined as the upper
hull of the union of $\mathcal{A}$ and $\mathcal{S}_\delta$.
\begin{figure}[H]
     \centering
     \begin{subfigure}[b]{0.3\textwidth}
         \centering
         \includegraphics[width=\textwidth]{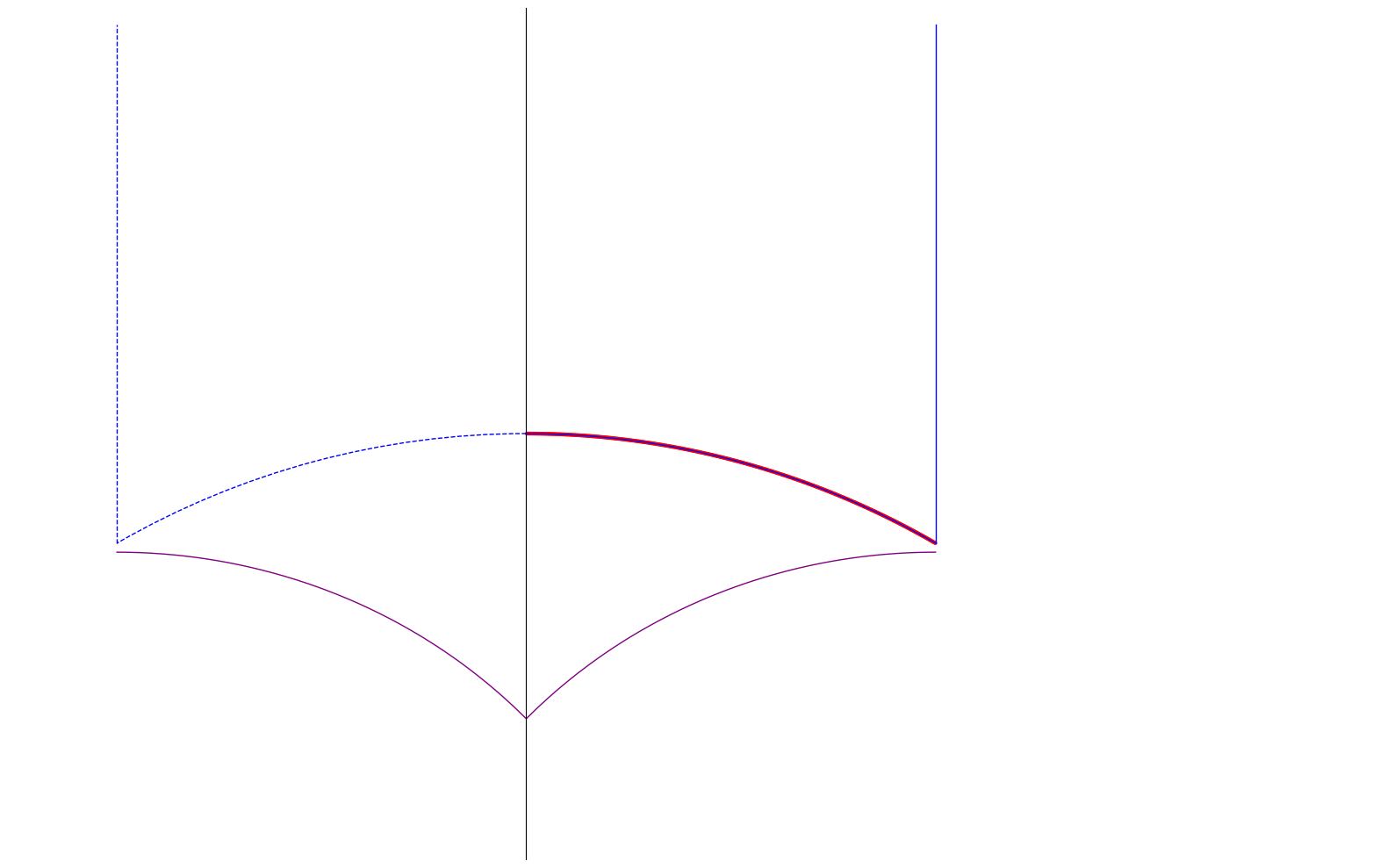}
         \subcaption{$\delta=0.6, k=12000$}
         \end{subfigure}
     \hfill
     \begin{subfigure}[b]{0.3\textwidth}
         \centering
         \includegraphics[width=\textwidth]{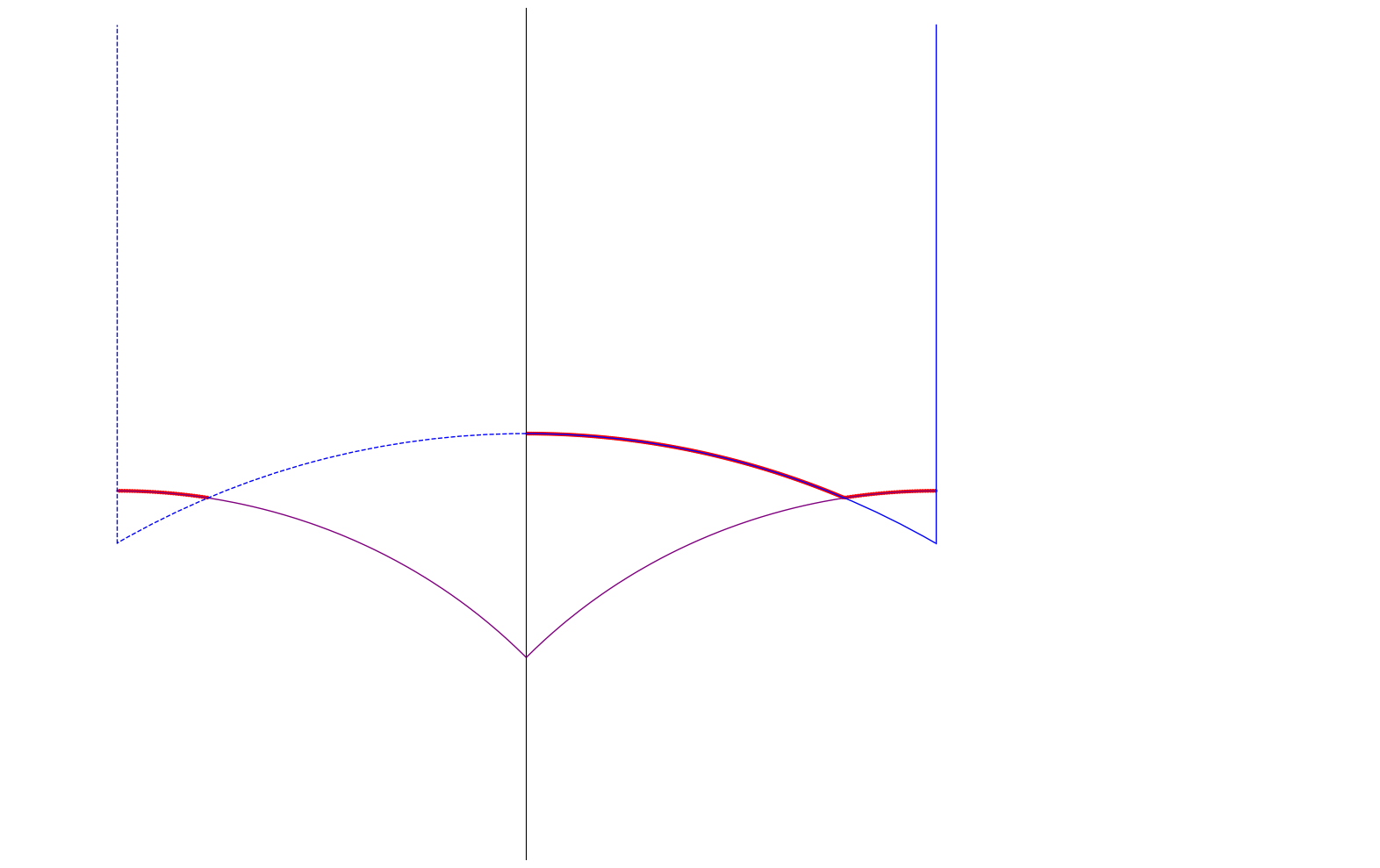}
         \subcaption{$\delta=0.75, k=12000$}
     \end{subfigure}
     \hfill
     \begin{subfigure}[b]{0.3\textwidth}
         \centering
         \includegraphics[width=\textwidth]{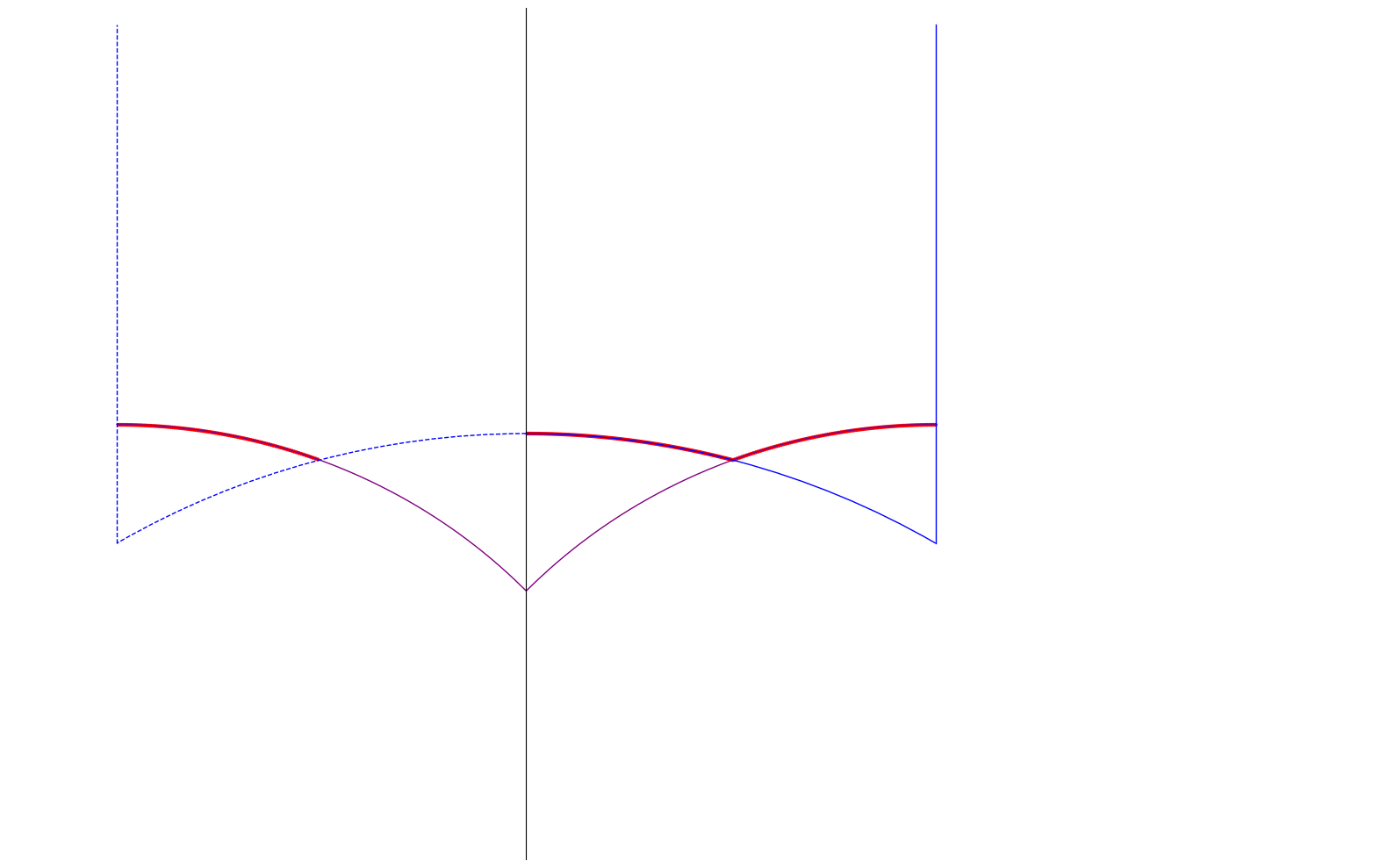}
         \subcaption{$\delta=0.85, k=24000$}
     \end{subfigure}
     \vspace{0.5cm} 

     \begin{subfigure}[b]{0.3\textwidth}
         \centering
         \includegraphics[width=\textwidth]{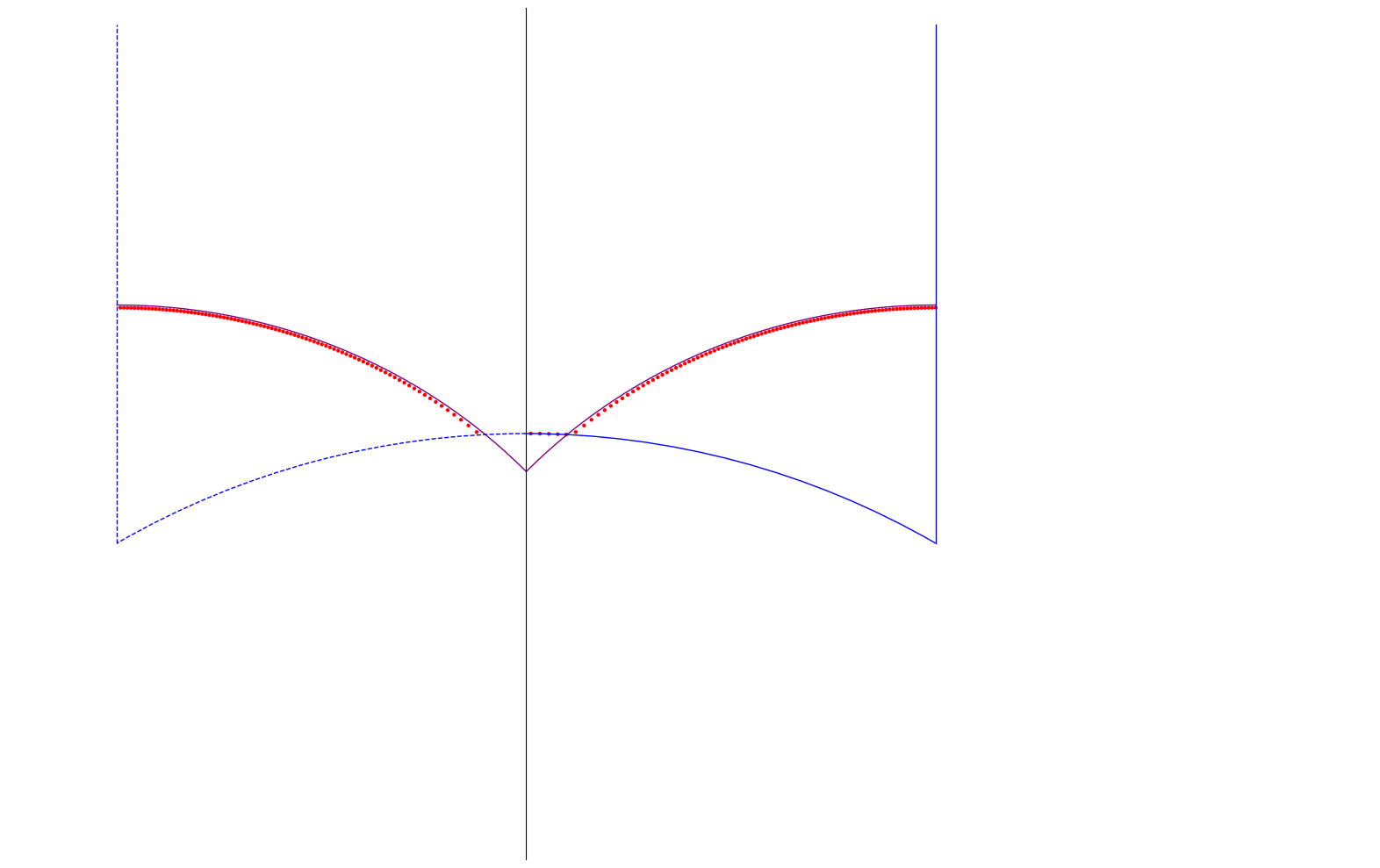}
         \subcaption{$\delta=0.94, k=36000$}
     \end{subfigure}
\hfill
     \begin{subfigure}[b]{0.3\textwidth}
         \centering
         \includegraphics[width=\textwidth]{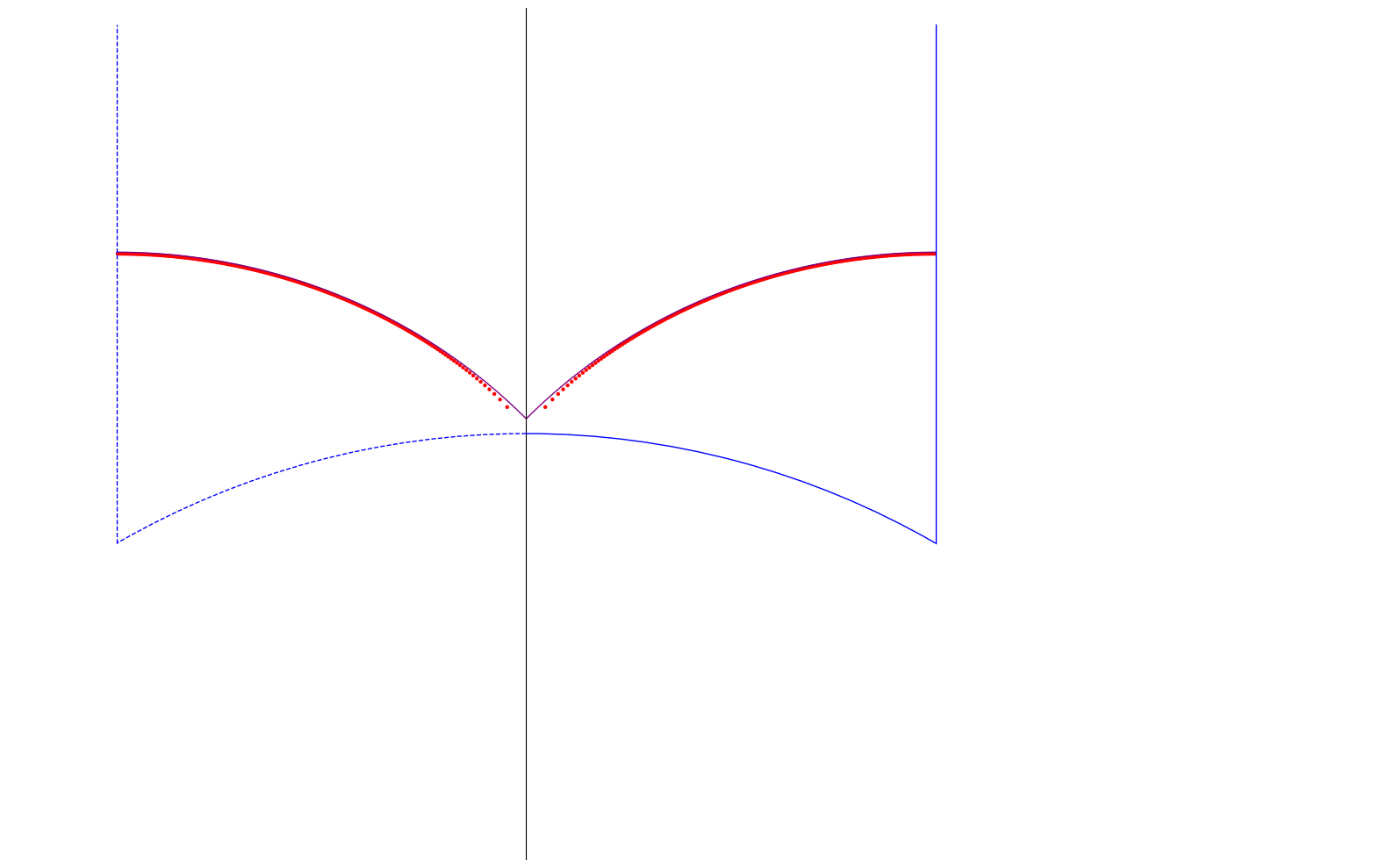}
         \subcaption{$\delta=0.96,k=120000$}
     \end{subfigure}
\hfill
     \begin{subfigure}[b]{0.3\textwidth}
         \centering
         \includegraphics[width=\textwidth]{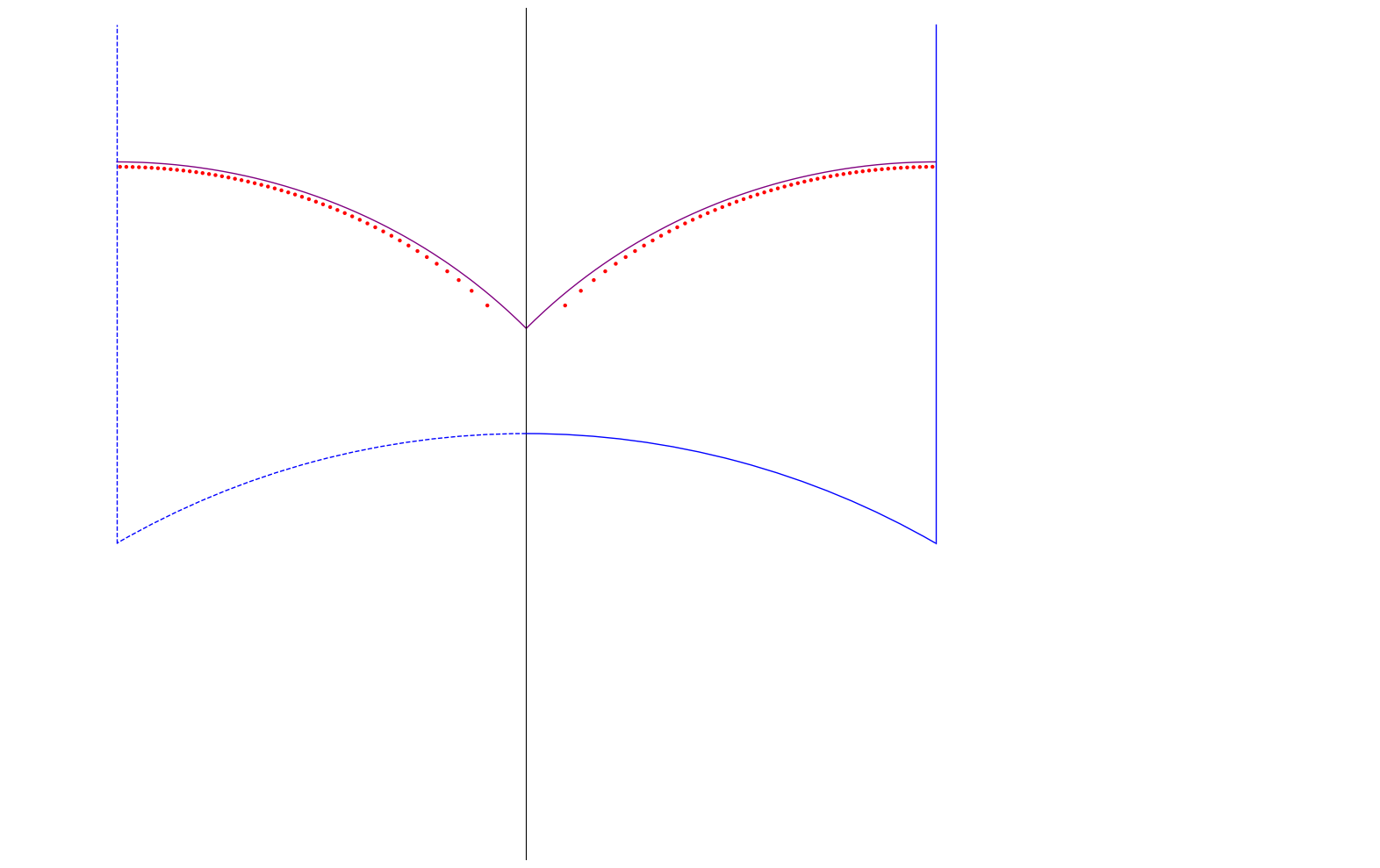}
         \subcaption{$\delta=0.98,k=240000$}
         \end{subfigure}
         \caption{Roots and $\mathcal{S}_\delta$ for several values of $\delta\in (0,1)$}
   \end{figure}
\begin{figure}[H]
     \centering
     \begin{subfigure}[b]{0.3\textwidth}
         \centering
         \includegraphics[width=\textwidth]{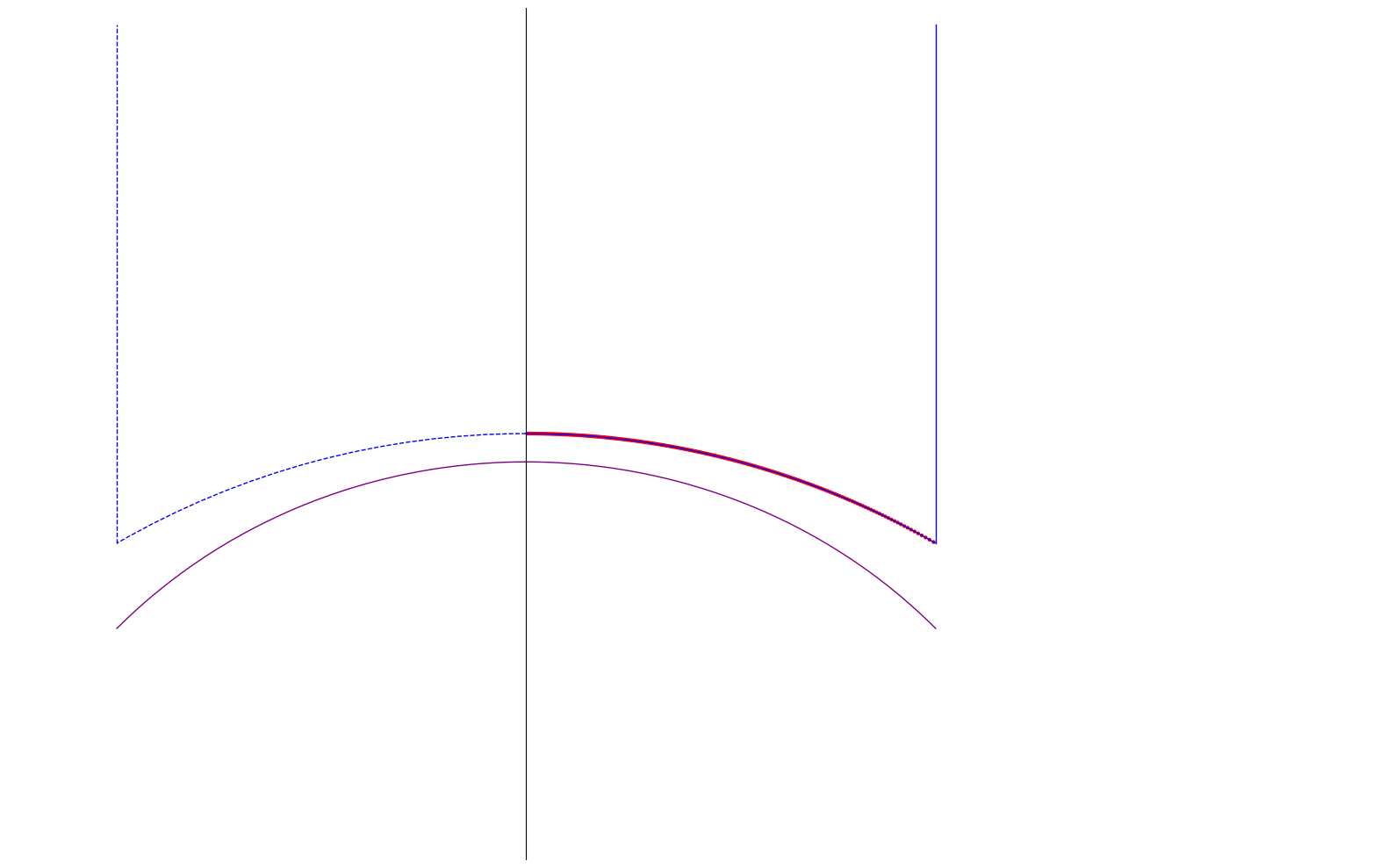}
         \subcaption{$\delta=1.2, k=-12000$}
     \end{subfigure}
     \hfill
     \begin{subfigure}[b]{0.3\textwidth}
         \centering
         \includegraphics[width=\textwidth]{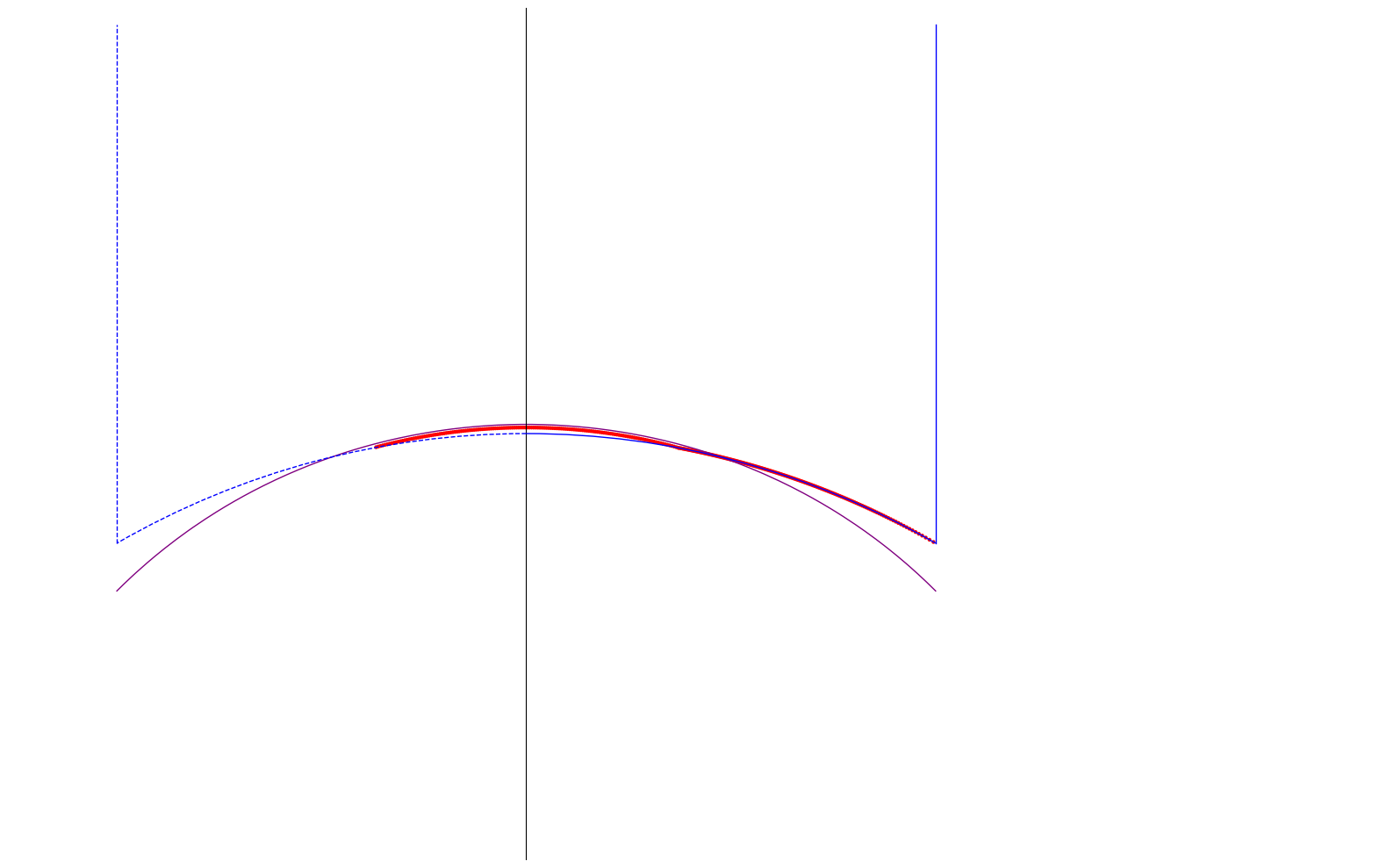}
         \subcaption{$\delta=1.15, k=-24000$}
     \end{subfigure}
     \hfill
     \begin{subfigure}[b]{0.3\textwidth}
         \centering
         \includegraphics[width=\textwidth]{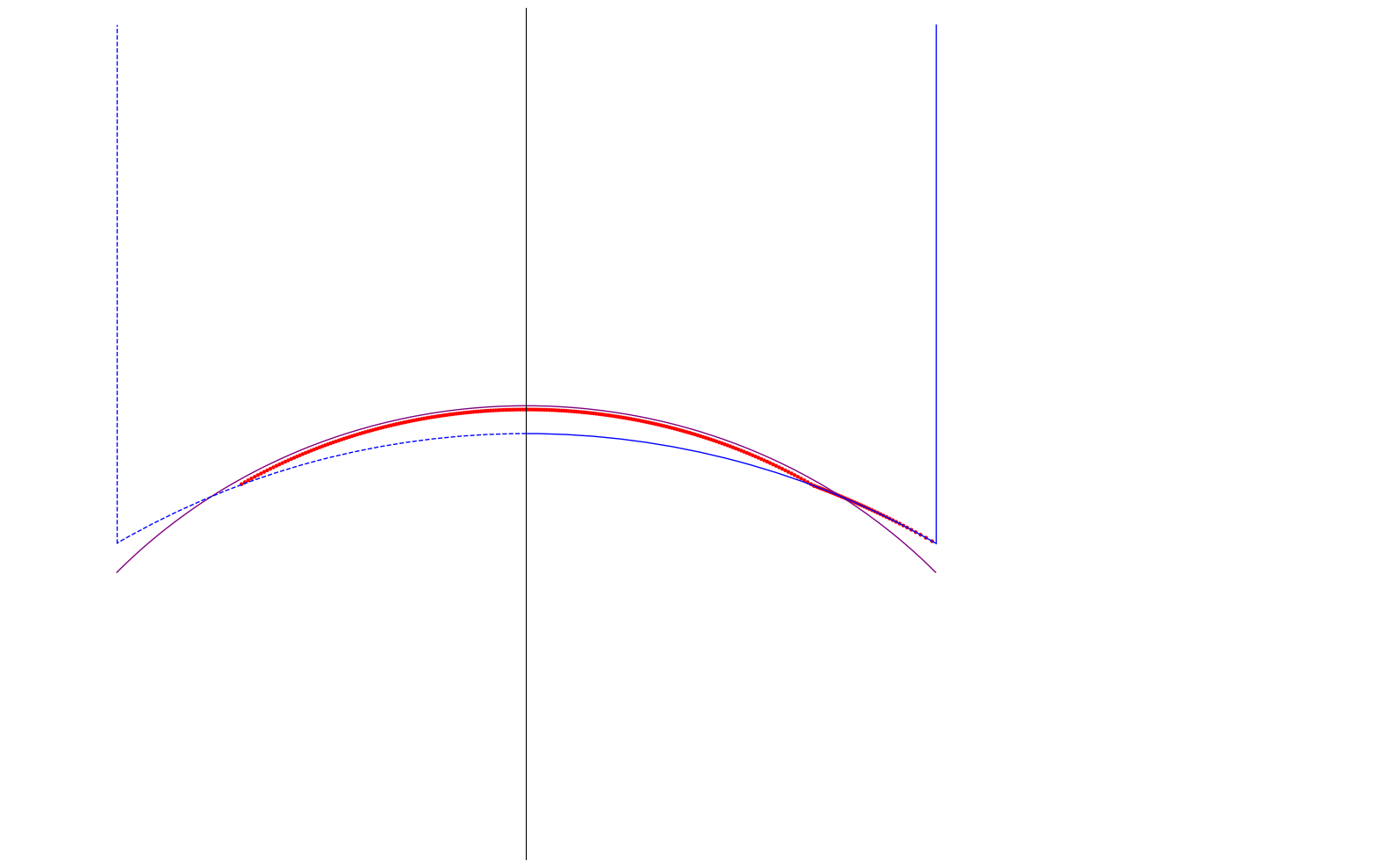}
         \subcaption{$\delta=1.13, k=-24000$}
     \end{subfigure}

     \vspace{0.5cm} 

     \begin{subfigure}[b]{0.3\textwidth}
         \centering
         \includegraphics[width=\textwidth]{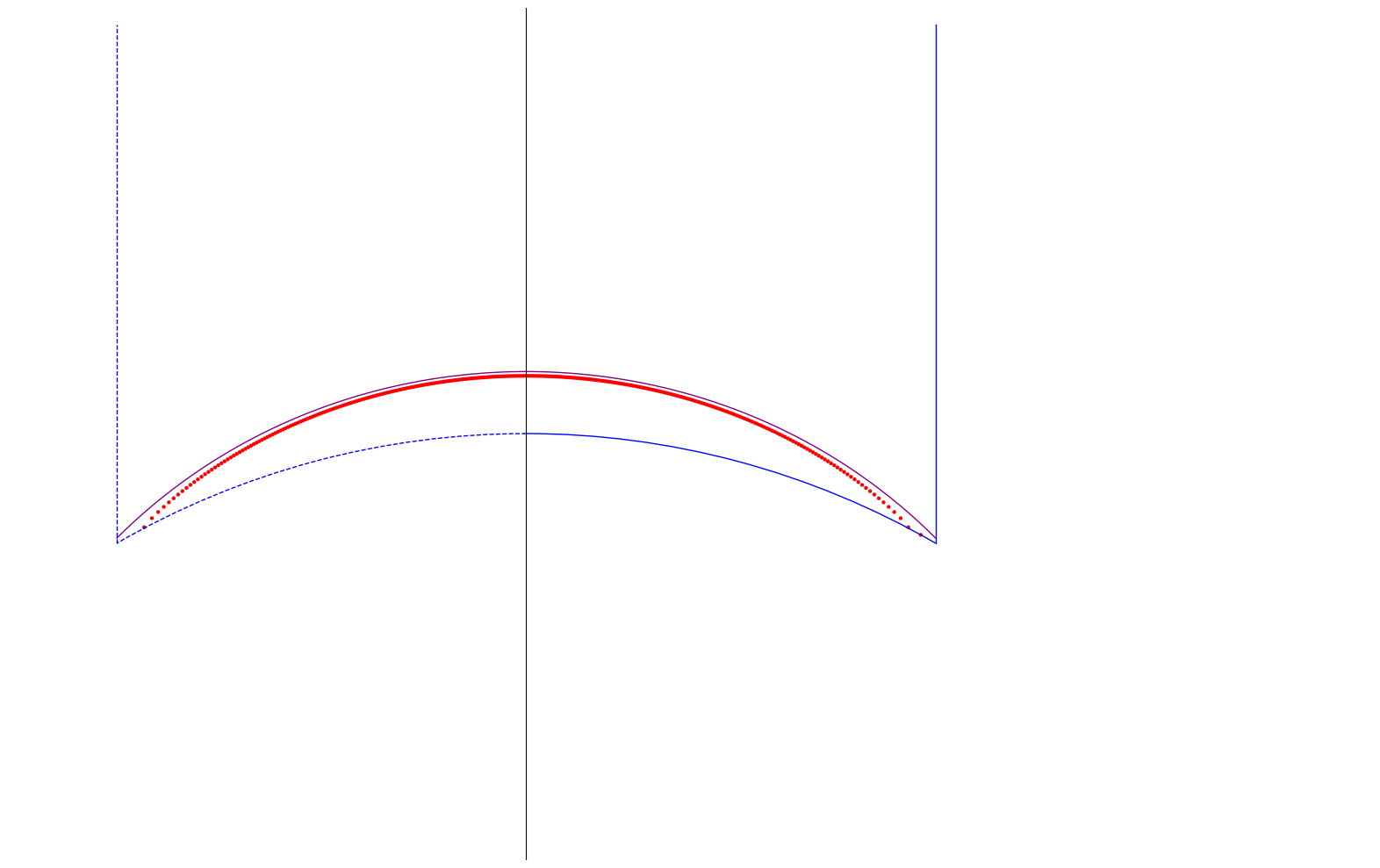}
         \subcaption{$\delta=1.1, k=-36000$}
     \end{subfigure}
     \hfill
     \begin{subfigure}[b]{0.3\textwidth}
         \centering
         \includegraphics[width=\textwidth]{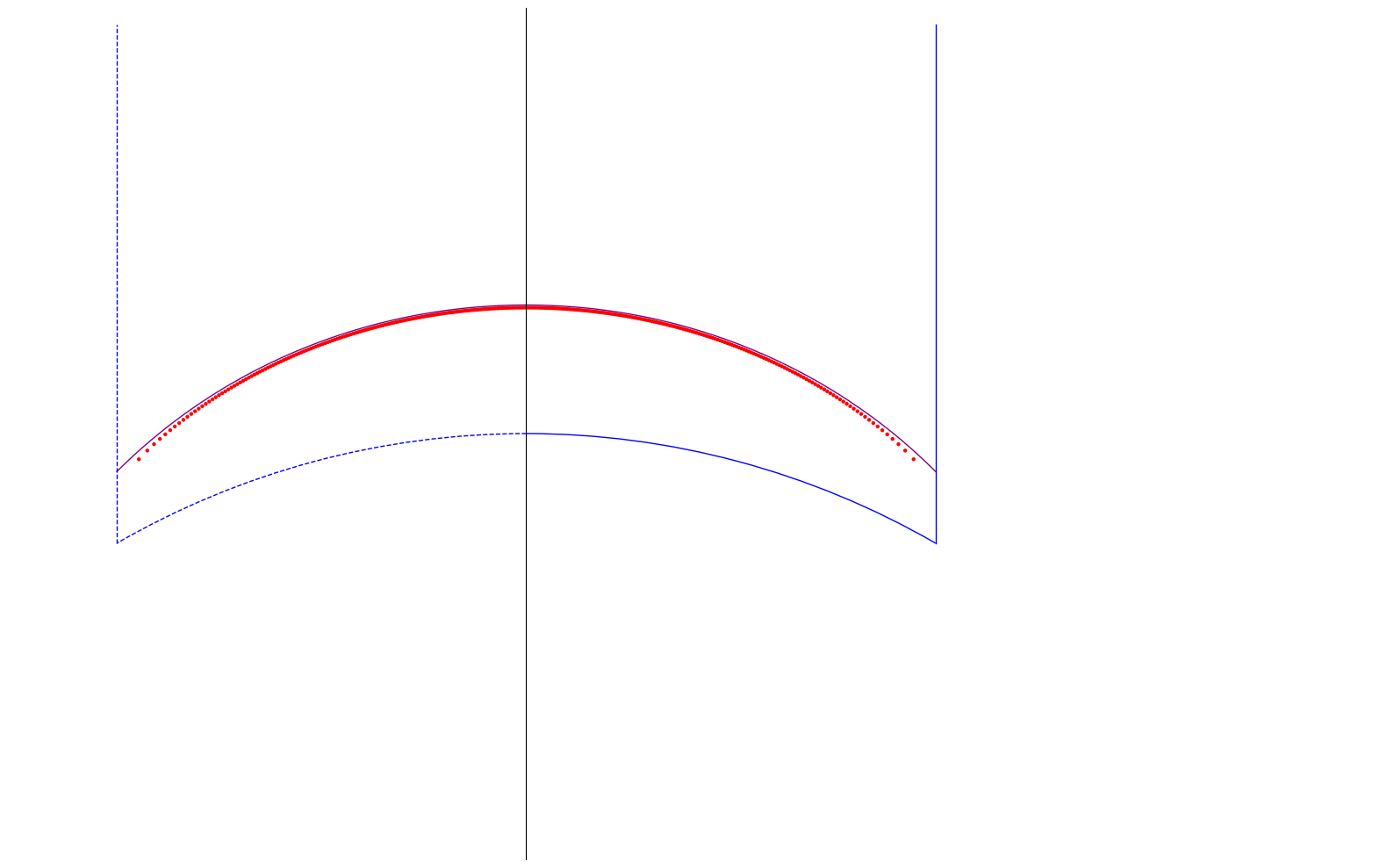}
         \subcaption{$\delta=1.06, k=-60000$}
     \end{subfigure}
     \hfill
     \begin{subfigure}[b]{0.3\textwidth}
         \centering
         \includegraphics[width=\textwidth]{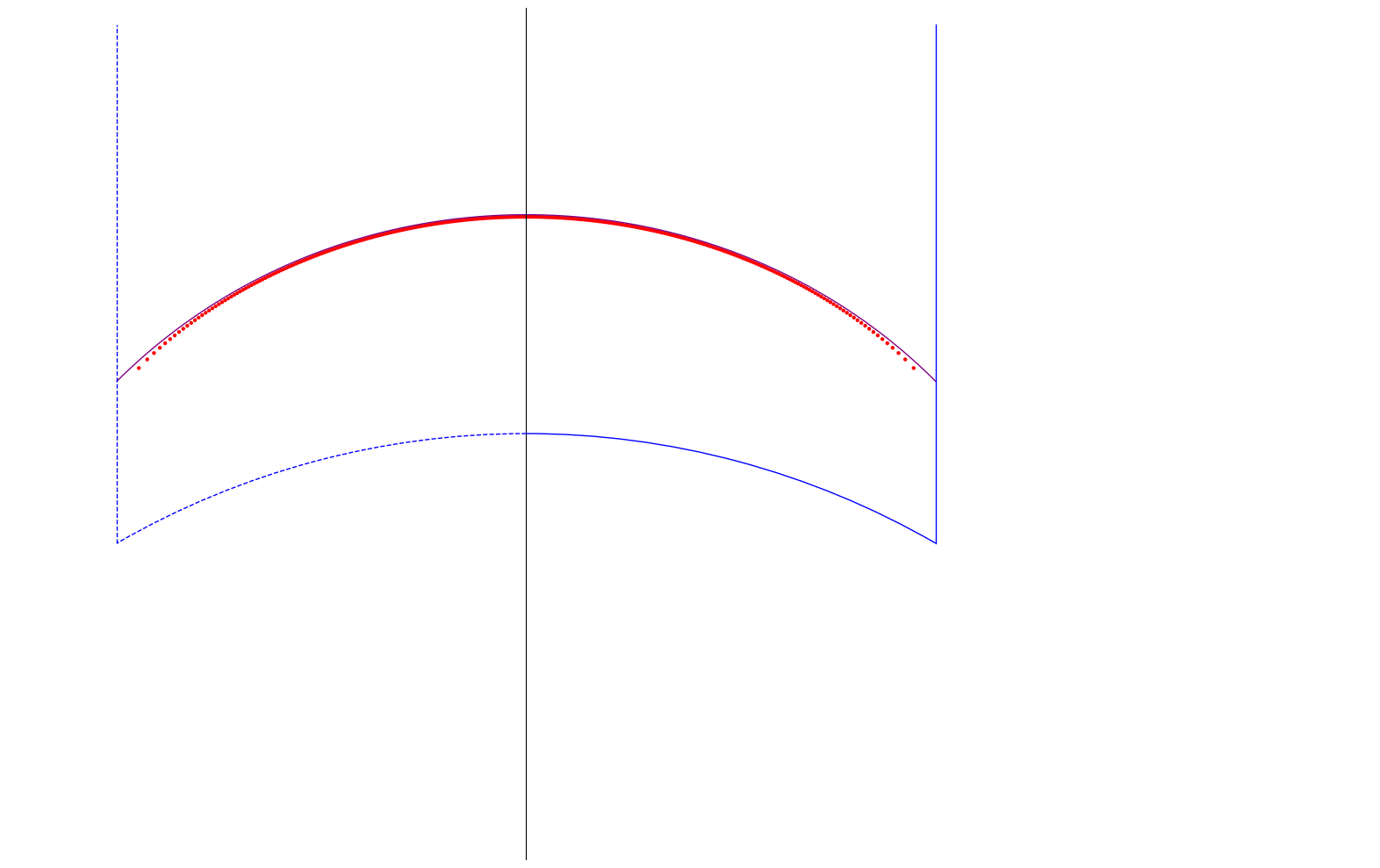}
         \subcaption{ for $\delta=1.03, k=-120000$}
         \end{subfigure}
         \caption{Roots and $\mathcal{S}_\delta$ for several values of $\delta\in (1,\infty)$}
     \end{figure}
\end{conjecture}

\begin{remark}\label{r:conjcutoffs}
Let
\begin{align*}
  \delta_{\mathcal{A}}^+&=1-\frac{24}{W(e^{-1})e^{\sqrt{3}\pi}}=0.6265\ldots&\delta_{\mathcal{A}}^-&=1+\frac{24}{W(e^{-1})e^{2\pi}}=1.1609\ldots\\
  \delta_{\mathcal{S}}^+&=1-\frac{24}{e^{2\pi}}=0.9551\ldots&\delta_{\mathcal{S}}^-&=1+\frac{24}{e^{\sqrt{3}\pi}}=1.1040\ldots
\end{align*}

The observation before Theorem B implies the following. 
\begin{enumerate}
\item If $\delta\in (0,1)$ then $\mathcal{C}_\delta=\mathcal{A}$
if $\delta<\delta_{\mathcal{A}}^+$ and
$\mathcal{C}_\delta=\mathcal{S}_\delta$ if
$\delta>\delta_{\mathcal{S}}^+$. If $\delta\in
(\delta_{\mathcal{A}}^+,\delta_{\mathcal{S}}^+)$, then
$\mathcal{C}_\delta$ contains
$\mathcal{A}_{\frac{\pi}{2},\mathcal{T}_\delta}$, where
$-\cos\mathcal{T}_\delta$ is the solution to
\[g_+(x)=\sqrt{1-x^2}+\frac{1}{2\pi}\log(1-\delta).\]
\begin{figure}[H]
\centering
\includegraphics[scale=0.15]{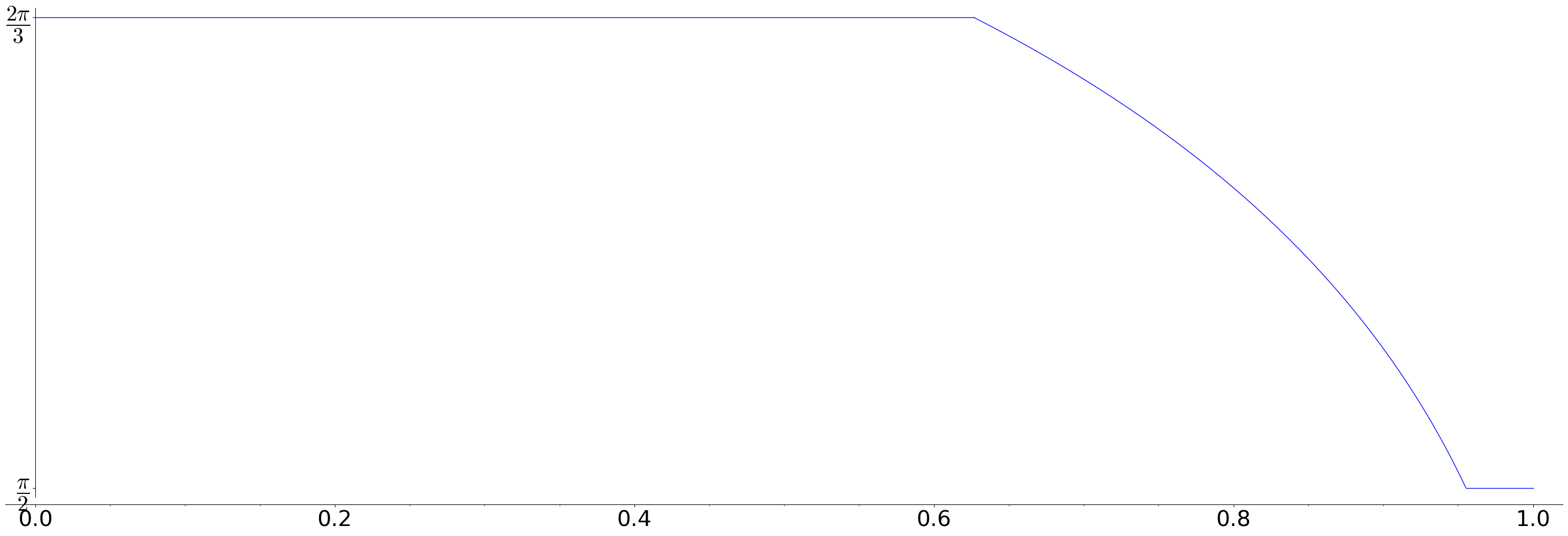}
\caption{Conjectural value of $\mathcal{T}_\delta$ with $\mathcal{A}_{\frac{\pi}{2},\mathcal{T}_-(\delta)}$ containing Miller zeros}
\end{figure}
The conjectural value of $\mathcal{T}_\delta$ is very close to the computational values of $\mathcal{T}(\delta)$ in Figure \ref{fig:calT1000}.
\item If $\delta\in (1,\infty)$ then $\mathcal{C}_\delta=\mathcal{A}$ if $\delta>\delta_{\mathcal{A}}^-$ and $\mathcal{C}_\delta=\mathcal{S}_\delta$ if $\delta<\delta_{\mathcal{S}}^-$. 
\end{enumerate}

\end{remark}
     
The results of the previous sections prove parts of this
Conjecture:
\begin{itemize}
\item By Theorem \ref{t:main-arc}, Conjecture \ref{conj} holds as $k\to \infty$ and
$\frac{m}{\ell}\to \delta<0.6194$. Moreover, if $\delta<0.9546$, the arc $\mathcal{A}_{\frac{\pi}{2},\mathcal{T}(\delta)}\subset \mathcal{C}_\delta$ contains some of the roots of $g_{k,m}$ for $\frac{m}{\ell}\to \delta$. 
\item By Theorem \ref{t:main-arc-weak}, Conjecture \ref{conj} holds as $k\to -\infty$ and
$\frac{m}{\ell}\to \delta>1.1598$. Moreover, if $\delta>1.1026$,
the arc $\mathcal{A}_{\mathcal{T}_-(\delta),\frac{2\pi}{3}}\subset \mathcal{C}_\delta$ contains
some of the roots of $g_{k,m}$ for $\frac{m}{\ell}\to \delta$.
\item Finally, by Theorem \ref{t:conj-tinyD}, Conjecture \ref{conj} holds if $|k|\to\infty$ and $D=\ell-m<\frac{\alpha\log |k|}{\log\log|k|}$ for some $\alpha\in (0,1)$. 
\end{itemize}
Thus, we successfully proved Conjecture \ref{conj} at the two
ends of the interval for $\delta$. When $\delta$ is close, but
not sufficiently close, to 1, we cannot account for the part of
the curve $\mathcal{C}_\delta$ that is not on the arc
$\mathcal{A}$, including showing that the curve
$\mathcal{C}_\delta$ goes up as $\delta$ increases to 1. We have
to contend with the following partial result. 

\begin{theorem}\label{t:logk}
If $k\to \infty$ then all the zeros of the Miller modular forms $g_{k,m}$ lie in the region $\{z\in \mathcal{F}\mid \Im z < \frac{2\log k}{1-c}+O(1)\}$, where $c=\frac{e^{2\pi}}{1728}$.
\end{theorem}

We will prove Theorem \ref{t:logk} using bounds on the coefficients of Faber polynomials. 

\subsection{Bounds on the coefficients of Faber polynomials}
When $D\leq k$, Faber polynomials are no longer approximated by truncated exponential polynomials, but we can still bound the growth of their coefficients. Indeed, we may rewrite \eqref{eq:faber} as 
\begin{align*}
  \begin{pmatrix} 1&0&0&0&0&0\\
  c_{D,-(D-1)}&1&0&0&0&0\\
  c_{D,-(D-2)}&c_{D-1,-(D-2)}&1&0&0&0\\
  c_{D,-(D-3)}&c_{D-1,-(D-3)}&c_{D-2,-(D-3)}&1&0&0\\
  &&&\vdots&&\\
  c_{D,0}&c_{D-1,0}&c_{D-2,0}&\ldots&c_{1,0}&1
  \end{pmatrix}\begin{pmatrix} y_0\\y_1\\\vdots\\y_D\end{pmatrix}=\begin{pmatrix} \mathcal{D}_{\ell,0}\\\mathcal{D}_{\ell,1}\\\vdots\\\mathcal{D}_{\ell,D}\end{pmatrix}.
\end{align*}
Inverting, we get
\begin{align*}
\begin{pmatrix} y_0\\y_1\\\vdots\\y_D\end{pmatrix}=  \begin{pmatrix} 1&0&0&0&0&0\\
  u_{D,-(D-1)}&1&0&0&0&0\\
  u_{D,-(D-2)}&u_{D-1,-(D-2)}&1&0&0&0\\
  u_{D,-(D-3)}&u_{D-1,-(D-3)}&u_{D-2,-(D-3)}&1&0&0\\
  &&&\vdots&&\\
  u_{D,0}&u_{D-1,0}&u_{D-2,0}&\ldots&u_{1,0}&1
  \end{pmatrix}\begin{pmatrix} \mathcal{D}_{\ell,0}\\\mathcal{D}_{\ell,1}\\\vdots\\\mathcal{D}_{\ell,D}\end{pmatrix},
\end{align*}
giving $\displaystyle y_d=\sum_{i=0}^d u_{D-i,-(D-d)}\mathcal{D}_{\ell,i}$,
where
\[u_{D-i,-(D-j)}=\sum_{D-j= i_s<i_{s-1}<\ldots <i_0=D-i}(-1)^s\prod_{r=0}^{s-1}c_{i_r,-i_{r+1}}.\]
Indeed, the inverse of lower unipotent matrices is given by $(I_m+N)^{-1}=\sum (-1)^s N^s$ where $N$ is nilpotent. We conclude that
\[|u_{D-i,-(D-j)}|\leq \sum_{D-j= i_s<i_{s-1}<\ldots
  <i_0=D-i}\prod_{r=0}^{s-1}c_{i_r,-i_{r+1}}.\]

We denote
\begin{align*}
  f(x,y)&=(x-y)\log(1728(1-c))+x\log x-y\log y -(x-y)\log(x-y)\\
  g(x,y)&=x\log 1728 - 2\pi y.
\end{align*}
Note that $f(x,cx)=g(x,cx)$ and $\lim\limits_{y\to x}f(x,y)=0$,
which means that
\[\BP(x,y)=\begin{cases}f(x,y)&y\geq cx\\g(x,y)&y<cx
\end{cases}
\]
is a continuous function on $\{(x,y)\mid x\geq y\geq
0\}$.

For convenience, we say a pair $(x,y)$ is good if $y\geq cx$ and is bad if $y<cx$. 
\begin{lemma}\label{l:BPdiff}
Given $z< x$, the function $h(y)=\BP(x,y)+\BP(y,z)$ on
the interval $(z,x)$ has the following monotonicity
\begin{center}
\begin{tabular}{llll}
  $(x,y)$&$(y,z)$&Increasing&Descreasing\\
  \hline
  good ($y\geq cx$)&good ($z\geq cy$)&$y< \frac{x+z}{2}$&$y> \frac{x+z}{2}$\\
  good ($y\geq cx$)&bad ($z< cy$)&$y< \frac{x}{2-c}$&$y> \frac{x}{2-c}$\\
  bad ($y< cx$)&good ($z\geq cy$)&always&\\
  bad ($y< cx$)&bad ($z< cy$)&always&
\end{tabular}
\end{center}
\end{lemma}
\begin{proof}
Follows from
\begin{align*}
\partial_y \left(f(x,y)+f(y,z)\right)&=\log(x-y)-\log(y-z)\\
\partial_y \left(f(x,y)+g(y,z)\right)&=\log(x-y)-\log y-\log(1-c)\\
\partial_y \left(g(x,y)+f(y,z)\right)&=\log y -\log(y-z)+\log\left(\frac{1-c}{c}\right)\\
\partial_y \left(g(x,y)+g(y,z)\right)&=-\log c.
\end{align*}

\end{proof}

\begin{lemma}\label{l:BP}
If $D-j=i_s<i_{s-1}<\ldots<i_0=D-i$ then
\[\prod_{r=0}^{s-1}c_{i_r,-i_{r+1}}\leq \max(T_1,T_2,T_2',T_3,T_3',T_4),\]
where
\begin{align*}
  T_1&=(1728-e^{2\pi})^{j-i}s^{j-i}\frac{(D-i)^{D-i}}{(D-j)^{D-j}(j-i)^{j-i}}\\
  T_2&=2981^{D-i}1.25^{D-j}\\
  T_2'&=(2393s)^{D-i}1.45^{j-i}\\
  T_3&=(2386s)^{D-i}2924^{-(D-j)}\\
  T_3'&=5849^{D-i}2924^{-(D-j)}\\
  T_4&=1728^{D-i}e^{-2\pi (D-j)}.
\end{align*}
When $D-j>c(D-i)$, it suffices to consider only $T_1$. In
particular, if $\mathcal{A}=5849$ and $\mathcal{B} = 1.5$ then
\[\prod_{r=0}^{s-1}c_{i_r,-i_{r+1}}\leq \max(T_1,(\mathcal{A}s)^{D-i}\mathcal{B}^{D-j}).\]
\end{lemma}
\begin{proof}
We will bound
\[\sum_{r=0}^{s-1} \log c_{i_r,-i_{r+1}}\leq
\sum_{r=0}^{s-1}\BP(i_r,i_{r+1}).\]
To estimate an upper bound, we will determine the maximum of the continuous function
\[\BP(x_0,x_1,\ldots,x_n)=\sum_{r=0}^{s-1}\BP(x_r,x_{r+1}),\]
in the compact set
\[\{(x_0,\ldots, x_n)\mid A=x_n\leq x_{n-1}\leq\ldots \leq x_0=B\}.\]
Let $(x_0,\ldots, x_n)$ be the point
where the maximum is attained.

Since $\BP(x,x)=0$, we may instead assume that the maximum is
attained at a point $(y_0,\ldots, y_m)$ such that
$A=y_m<y_{m-1}<\ldots<y_0=B$, where $m\leq n$, by eliminating
duplicates.

We will represent the sequence $A=y_m<y_{m-1}<\ldots <y_0=B$ as
a ladder



\begin{center}
\includegraphics[scale=0.5]{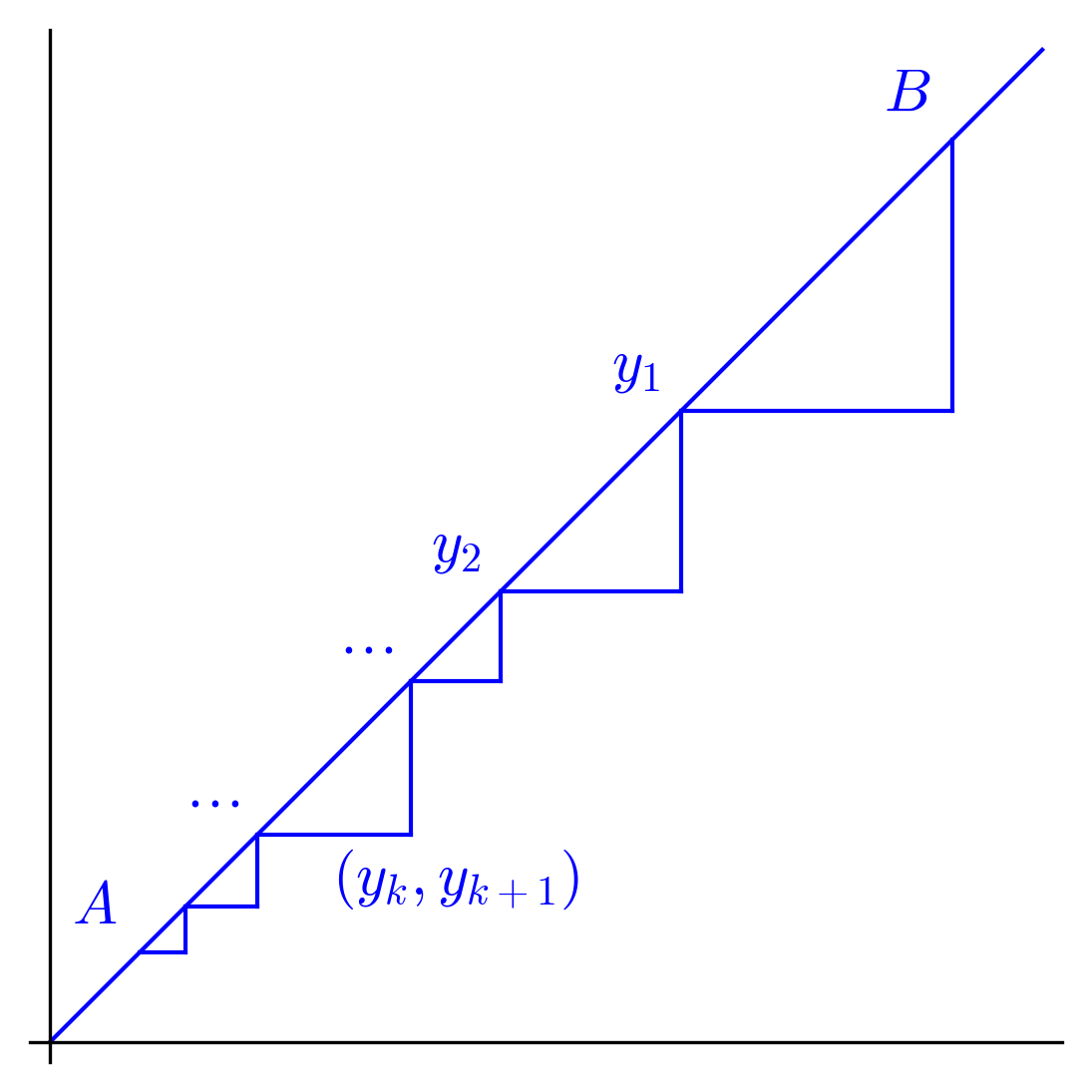}
\end{center}

To describe $(y_0,\ldots, y_m)$ at the maximum, we will
repeatedly use Lemma \ref{l:BPdiff}. We denote $\mathcal{L}$ the line $y=cx$ and $\mathcal{L}'$ the line $y=\frac{1}{2-c}x$.

Suppose that $P_k=(y_k,y_{k+1})$ is bad. First, note that $P_{k-1}=(y_{k-1},y_k)$
can't also be bad, since increasing
$y_k$ would increase the function. Thus 
$P_{k-1}$ is good and located on the line $\mathcal{L}'$, or else increasing or decreasing
$y_k$ would increase the function.

Suppose $P_k=(y_k,y_{k+1})$ is good, but not on the line
$\mathcal{L}$. If $P_{k-1}$ is bad, we could increase $y_k$ to increase
the value of the function, which means $P_{k-1}$ is also
good. If $P_{k-1}$ is on the line $\mathcal{L}$, then
$y_{k-1}+y_{k+1}>\frac{1}{c}y_k>2y_k$, so we may increase $y_k$
keeping both $P_k$ and $P_{k-1}$ good to increase the value of
the function. If $P_{k-1}$ is not on the line $\mathcal{L}$, it must be
that $2y_k=y_{k-1}+y_{k+1}$. We conclude that every good point not on the line $\mathcal{L}$ must be preceeded by evenly spaced out good points, also not on the line $\mathcal{L}$.

Finally, if $P_k$ is on the line $\mathcal{L}$, by the same
argument, the previous point $P_{k-1}$ must be good with
$y_{k-1}-y_k\geq y_k-y_{k+1}$.

Therefore, there are four possible configurations:
\begin{enumerate}
\item All points $P_0,\ldots, P_{m-1}$ are good, not on
$\mathcal{L}$, with $y_0,\ldots, y_m$ evenly spaced out.
\item The last point $P_{m-1}$ is bad, $P_{m-2}$ is on
$\mathcal{L}'$, and $y_0,\ldots, y_{m-1}$ are evenly spaced out.
\item The points $P_0,\ldots, P_{q-1}$ are good, not on
$\mathcal{L}$, with $y_0,\ldots, y_q$ evenly spaced out, and
$P_q,\ldots, P_{m-1}$ are all on $\mathcal{L}$.
\item The last, degenerate, case is $m=1$, $y_0=B$, $y_1=A$.
\end{enumerate}




\begin{figure}[htp]
\centering
\includegraphics[width=.3\textwidth]{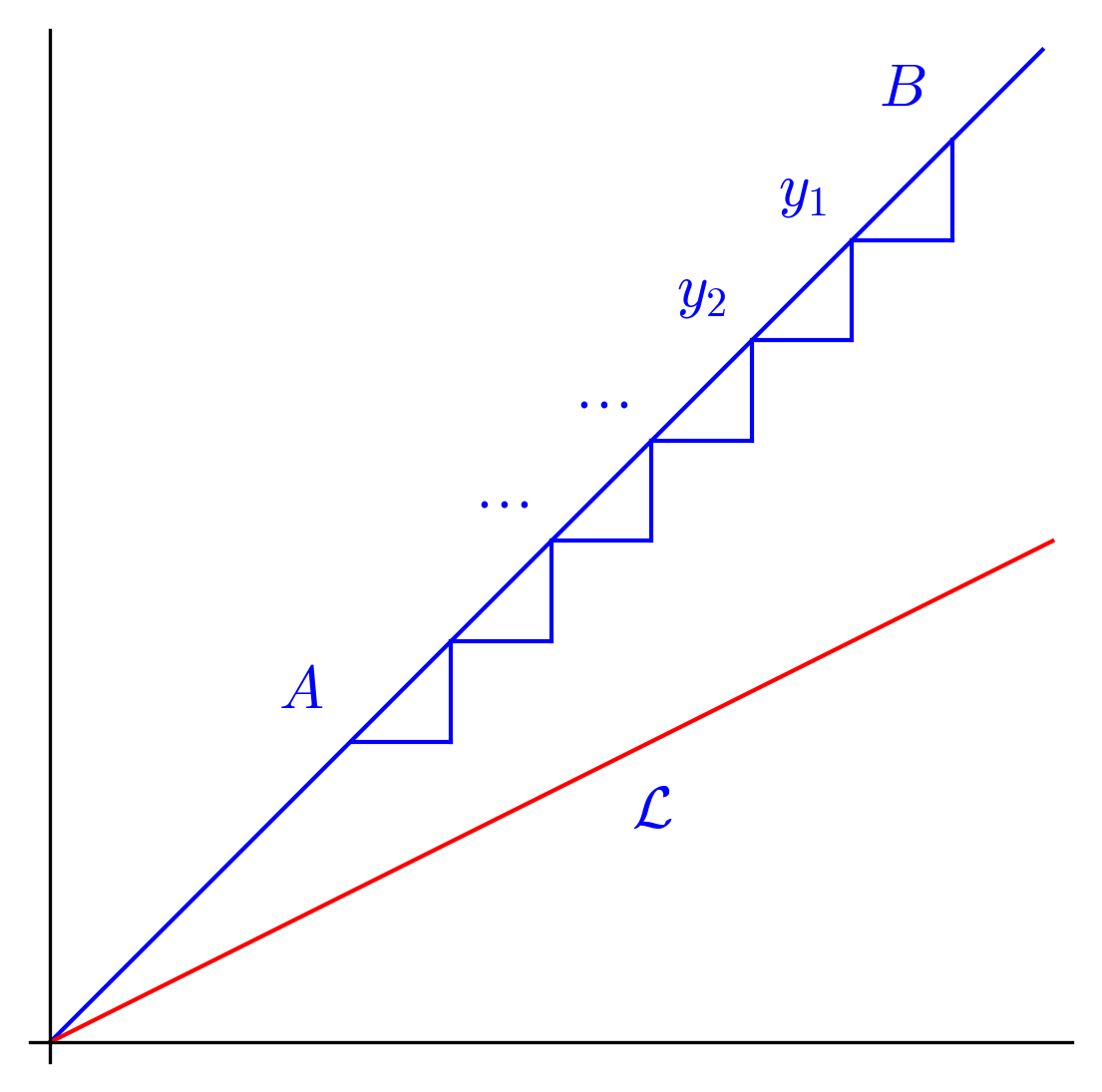}\hfill
\includegraphics[width=.3\textwidth]{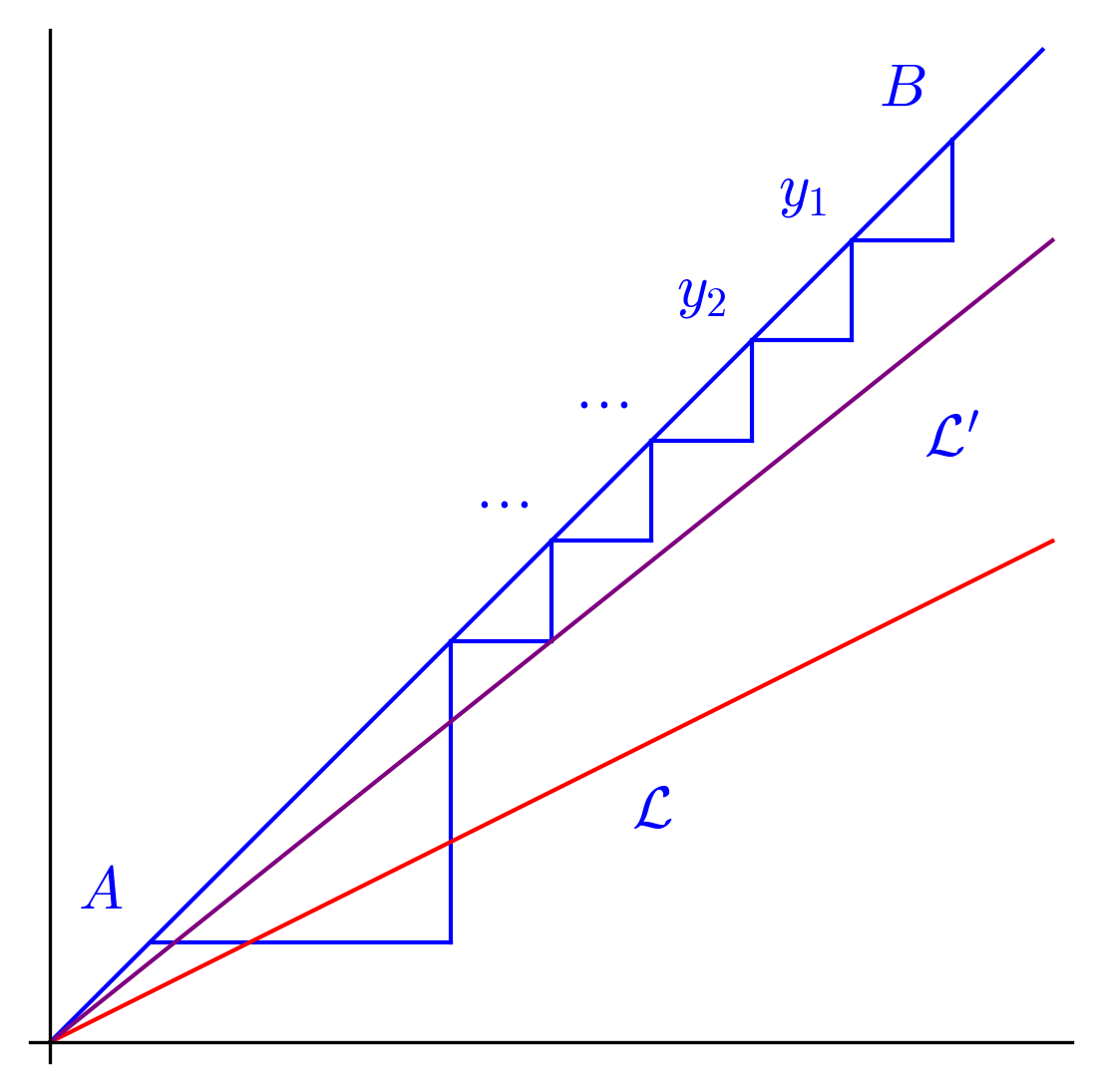}\hfill
\includegraphics[width=.3\textwidth]{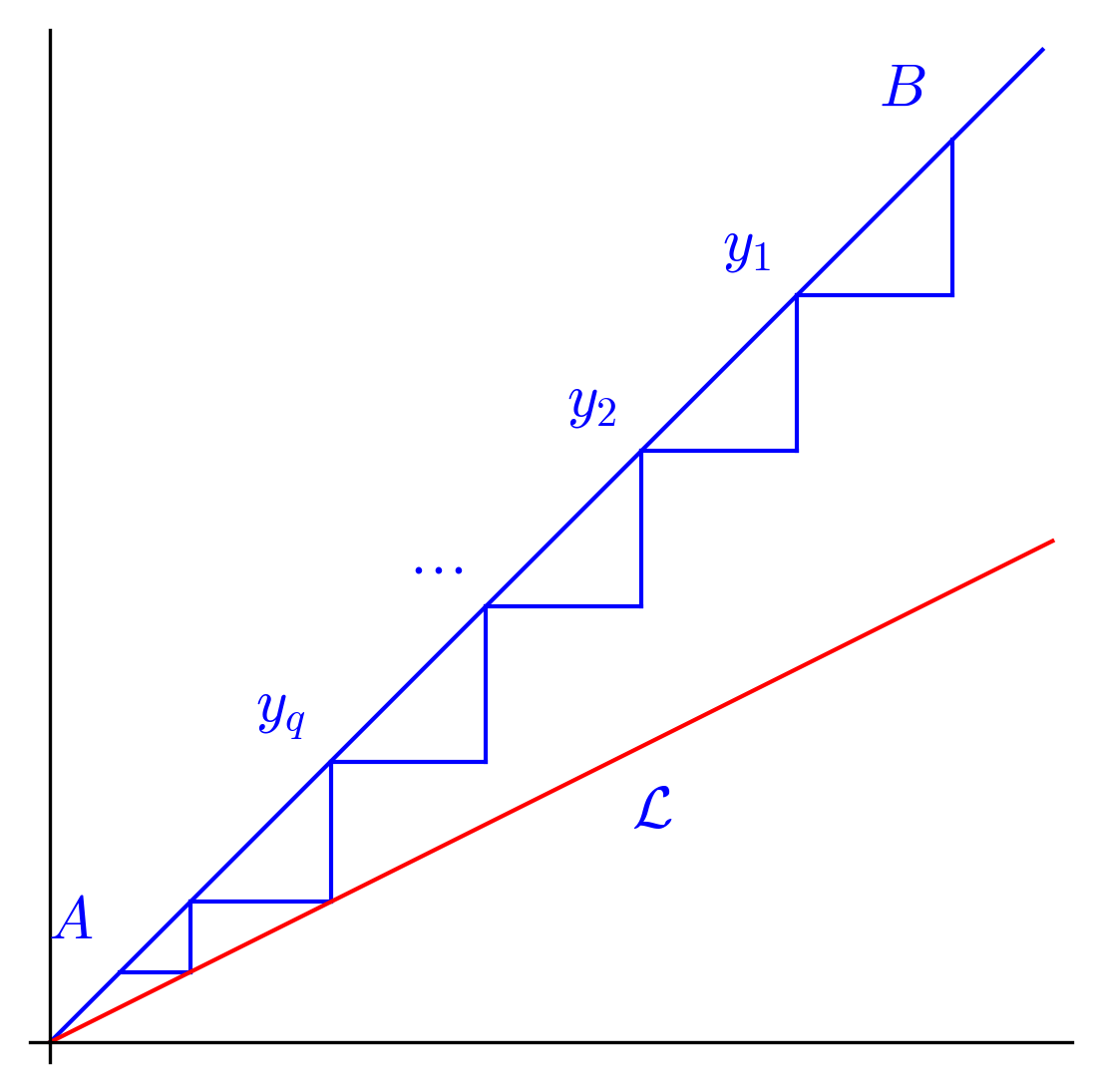}
\end{figure}

In each of these cases, we can determine $y_0,\ldots, y_m$
exactly. For convenience, we denote $[x,y]=x\log x-y\log y-(x-y)\log(x-y)\leq x\log 2$.
\begin{enumerate}
\item For $1\leq u\leq m$, $y_i = B-(B-A)i/m$. In this case, the
value of the function is
\begin{align*}
\BP(y_0,\ldots,y_m)&=\sum_{i=0}^{m-1}\left((y_i-y_{i+1})\log(1728(1-c))+y_i\log
                     y_i-y_{i+1}\log y_{i+1}-\frac{B-A}{m}\log
                     \frac{B-A}{m}\right)\\
  &=(B-A)(\log m+\log(1728(1-c)))+B\log B-A\log A-(B-A)\log (B-A)\\
  &=(B-A)(\log m+\log(1728(1-c)))+[B,A],
\end{align*}
which is maximized when $m=n$.
\item The remaining three cases can only occur if $A< \mathcal{B}$. We have $y_m=A$, $y_{m-1}=A+a$, $y_{m-2}=A+a+b$, \ldots,
$y_0=B=A+a+(m-1)b$ with $y_{m-1}=\frac{1}{2-c}y_{m-2}$. This
gives $b=(1-c)(A+a)$ and
\[a=\frac{B+A}{1+(m-1)(1-c)}-A.\]
In this case, denoting $N=1+(m-1)(1-c)$ the
value of the function is
\begin{align*}
g(A+a,A) &+\BP(y_0,\ldots,y_{m-1})\\ &=(A+a)\log 1728-2\pi
                                  A+(B-A-a)(\log(m-1)+\log(1728(1-c)))\\
                                &\ \ +B\log B-(A+a)\log (A+a)-(B-A-a)\log (B-A-a)\\
                                &=B(\log(1728)+\log(N-1))-(A+a)\log
                   (N-1)+[B,A+a]\\
  &\leq B(\log 1728+\log 2)+(B-\frac{A+B}{N})\log(N-1).
\end{align*}
When $m=2$, $N=2-c$ and the quantity above is
\[\leq B\left(\log 1728+\log 2+\left(1-\frac{1}{2-c}\right)\log(1-c)\right)-\frac{\log(1-c)}{2-c}A\leq 8B+0.22A.\]
When $m\geq 2$, $N=1+(1-c)(m-1)>2$ and the quantity above is
maximized when $m=n$ giving
\[\leq B\left(\log 1728+\log
2+\left(1-\frac{2}{N}\right)\log(N-1)\right)+(B-A)\frac{\log(N-1)}{N-1}\]
\[<B\left(\log 1728+\log
2+\log(1-c)+\log(n)\right)+e^{-1}(B-A)<(7.78+\log n)B+e^{-1}(B-A).\]

\item For $0\leq i\leq m-q$ we have $y_{m-i}=c^{-i}A$, and $y_j=B-(B-c^{q-m}A)j/(m-q)$ for $0\leq j\leq q$. In this case, the
value of the function is
\[A \left(\frac{c^{q-m}-1}{c^{-1}-1}\right)\left(c^{-1}\log
1728-2\pi\right)+(B-c^{q-m}A)(\log(q)+\log(1728(1-c)))+\left[B,c^{q-m}A\right]\]
\[\leq B\log 2 + A \left(\frac{c^{q-m}-1}{c^{-1}-1}\right)\left(c^{-1}\log
1728-2\pi\right)+(B-c^{q-m}A)(\log(q)+\log(1728(1-c)))\]
\[= B\log 2 - A \left(\frac{\log
  1728-2\pi c}{1-c}\right)+B(\log(q)+\log(1728(1-c)))+\]
\[+c^{q-m}A\left(\frac{(\log
1728-2\pi) c}{1-c}-(\log(q)+\log(1-c))\right)\]
Note that the last term is negative if $q\geq 3$. Keeping in
mind that $Ac^{q-m}\leq B$ we get that this is
\[\leq B\log 2 - A \left(\frac{\log
1728-2\pi c}{1-c}\right)+B(\log(q)+\log(1728(1-c)))\]
\[\leq B\left(\log 2 +\log(1728(1-c))+\log n\right)- A \left(\frac{\log
1728-2\pi c}{1-c}\right)\leq (7.78+\log n)B-7.98A\]
if $q\geq 3$ and 
\[\leq B\left(\log 2+\frac{\log
1728-2\pi c}{1-c}\right) - A \left(\frac{\log
1728-2\pi c}{1-c}\right)\leq 8.68B -7.98A\]
if $q\leq 2$.

We remark that this case can only occur if $A\leq cB$.
\item In the degenerate case when $y_0=B$ and $y_1=A$ the value
of the function is
\[B\log 1728-2\pi A.\]
\end{enumerate}

\end{proof}

\begin{proposition}\label{p:faber-coefficients-general}
We have $\displaystyle y_d=O\left(\max(U_1,U_2,U_3)\right)$,
where 
\begin{align*}
  U_1&=\left(\frac{2e \alpha k}{d}\right)^d\\
  U_2&=\frac{(\alpha k)^R(2e)^d(1728-e^{2\pi})^{d-R}(D-R)^{D-R}}{R^R(D-d)^{D-d}}\\
  U_3&=\left(\alpha \mathcal{B}^ckd)\right)^{\frac{d}{1-c}}.
\end{align*}
Here $\alpha=\left(\frac{25}{24}\right)^{25}$, $U_3$ is considered only when $d>(1-c)D$, and $U_2$ only when $R=\frac{1}{2}(D-\sqrt{D^2-4k/C})$ is real with $R<d$, where $C=\alpha^{-1}e^2(1728-e^{2\pi})$.
\end{proposition}
\begin{proof}
We know that
\[y_d=\sum_{r=0}^{d} u_{D-r,-(D-d)} \mathcal{D}_{\ell,r},\]
with
$\mathcal{D}_{\ell,d}=O\left(\left(\frac{2 \alpha
  |k|}{d}\right)^d\right)$ by Lemma \ref{l:rudnick-coeffs},
where $\alpha=\left(\frac{25}{24}\right)^{25}$. Here,
\[|u_{D-r,-(D-d)}|\leq
\sum_{D-d=i_s<i_{s-1}<\ldots<i_0=D-r}\prod c_{i_u,i_{u+1}}\leq
\sum_{s=1}^{d-r}\frac{(d-r)^s}{s!}\max \prod c_{i_u,i_{u+1}}.\]
Whenever $D-d>c(D-r)$ (which is always guaranteed when $d<D(1-c)$), we know from Lemma \ref{l:BP} that
\[\max \prod c_{i_u,i_{u+1}}\leq
(1728-e^{2\pi})^{d-r}s^{d-r}\frac{(D-r)^{D-r}}{(D-d)^{D-d}(d-r)^{d-r}}.\]
Thus, for $d<D(1-c)$, we have
\begin{align*}
|y_d|&=O\left(\sum_{r=0}^{d}\left(\frac{2 \alpha k}{r}\right)^r(1728-e^{2\pi})^{d-r}\frac{(D-r)^{D-r}}{(D-d)^{D-d}(d-r)^{d-r}}\sum_{s=1}^{d-r}\frac{(d-r)^s}{s!}
                                              s^{d-r}\right)\\
  &=O\left(\sum_{r=0}^{d}\left(\frac{2 \alpha k}{r}\right)^r(1728-e^{2\pi})^{d-r}\frac{(D-r)^{D-r}}{(D-d)^{D-d}}(d-r)e^{d-r}\right)\\
  &=O\left(\sum_{r=0}^{d}\frac{(\alpha k)^{r}(2e)^d}{r^r}(1728-e^{2\pi})^{d-r}\frac{(D-r)^{D-r}}{(D-d)^{D-d}}\right),
\end{align*}
since $d-r\leq 2^{d-r}$.
The derivative with respect to $r$ of log of the term above is
\[\log (\alpha k) - \log r - 2 -\log(1728-e^{2\pi})-\log (D-r).\]
This is positive whenever $r(D-r)\leq k/C$, where
$C=\alpha^{-1}e^2(1728-e^{2\pi})$. This is always true if
\begin{enumerate}
\item $D^2-4k/C<0$ or
\item $d<R=\frac{1}{2}(D-\sqrt{D^2-4k/C})$.
\end{enumerate}
In this case, the term in the sum is maximized when $r=d$ and
the sum is
\[|y_d|=O\left(\left(\frac{2e \alpha k}{d}\right)^d\right).\]
If $D^2>4k/C$ and $d>R$, the terms in the sum attains its
maximum either at $r=R$ or at $r=d$
and we obtain
\[|y_d|=O\left(\max \left(\left(\frac{2e \alpha k}{d}\right)^d, \frac{(\alpha k)^R(2e)^d(1728-e^{2\pi})^{d-R}(D-R)^{D-R}}{R^R(D-d)^{D-d}}\right)\right).\]

If $d>D(1-c)$, we have to break up the sum into pieces based on $D-r$. We use the bounds from Lemma \ref{l:BP}.
\begin{align*}
  |y_d|&=O\left(\sum_{r<D-c^{-1}(D-d)}\left(\frac{2 \alpha k}{r}\right)^r\mathcal{A}^{D-r}\mathcal{B}^{D-d}\sum_{s=0}^{d-r}\frac{(d-r)^ss^{D-r}}{s!}\right)\\
  &\
    +O\left(\sum_{r>D-c^{-1}(D-d)}^{d}\left(\frac{2 \alpha k}{r}\right)^r(1728-e^{2\pi})^{d-r}\frac{(D-r)^{D-r}}{(D-d)^{D-d}(d-r)^{d-r}}\sum_{s=1}^{d-r}\frac{(d-r)^s}{s!}s^{d-r}\right)\\
  &=
    O\left(\sum_{r<D-c^{-1}(D-d)}\frac{(\alpha k)^{r}(2e)^d}{r^r}\mathcal{A}^{D-r}\mathcal{B}^{D-d}(d-r)^{D-r}\right)\\
  &\ +O\left(\sum_{r>D-c^{-1}(D-d)}^{d}\frac{(\alpha k)^{r}(2e)^d}{r^r}(1728-e^{2\pi})^{d-r}\frac{(D-r)^{D-r}}{(D-d)^{D-d}}\right).
\end{align*}
The second sum has the same upper bounds as above. For the first sum, note that $(\alpha k \mathcal{A}^{-1})^r\leq (\alpha k \mathcal{A}^{-1})^{D-c^{-1}(D-d)}<(\alpha k \mathcal{A}^{-1})^{d/(1-c)}$ (if $\alpha k>\mathcal{A}$) and $r^{-r}(d-r)^{D-r}\leq d^D$, yielding the bound $O((\alpha \mathcal{B}^c kd)^{d/(1-c)})$.
\end{proof}
\subsection{Proof of Theorem \ref{t:logk}}
Suppose $g_{k,m}(z)=0$ for some $z\in \mathcal{F}$. Then $j(z)$
is a root of the Faber polynomial $\sum_{d=0}^D y_{D-d}x^d$ and
thus, by Lagrange's inequality
\[|j(z)|\leq 2\max (|y_d|^{1/d}).\]
By Proposition \ref{p:faber-coefficients-general}, we have
\[|y_d|^{1/d}=O\left(\max\left(\frac{2e \alpha k}{d},
(\alpha \mathcal{B}^c kd)^{\frac{1}{1-c}},
U_2^{1/d}\right)\right),\]
where
\begin{align*}
  U_2&=\frac{(\alpha k)^R(2e)^d(1728-e^{2\pi})^{d-R}(D-R)^{D-R}}{R^R(D-d)^{D-d}}
\end{align*}
is counted only when $D^2\geq 4k/C$ and
$R=\frac{1}{2}(D-\sqrt{D^2-4k/C})<d$. We see that
\begin{align*}
\frac{\log U_2}{d}&\leq \log(2e \alpha(1728-e^{2\pi}))+\frac{(D-R)\log(D-R)-(D-d)\log(D-d)}{d}.
\end{align*}
By the Mean Value Theorem, the latter equals $1+\log T$ for some
$T\in (D-d,D)$ and so
$\displaystyle \frac{\log U_2}{d}\leq \log(2e^2 \alpha(1728-e^{2\pi}))+\log(kD)$.
Therefore,
\[|j(z)|=\max(O(k), O((kD)^{\frac{1}{1-c}}), O(kD)))=O(k^{\frac{2}{1-c}}).\]
Finally, if $|j(z)|\gg 0$ then $\Im z=\log |j(z)|+O(1)$ and the conclusion follows.

\section{Algebraic zeros} \label{Algebraic-zeros}

As noted in the introduction, Kohnen \cite{kohnen:eisenstein} proved that the only possible algebraic zeros of $E_k$ in $\mathcal{F}$ are $\rho$ or $i$, a conclusion that ultimately rests on the fact that all zeros of $E_k$ lie on $\mathcal{A}$. Since the Miller forms $g_{k,m}$ do not, in general, share this feature, the possibility of additional algebraic zeros arises. While a specific instance of this phenomenon for weakly holomorphic  Miller forms has been previously observed in \cite{swisher}, a systematic investigation has not been undertaken. In this section, we obtain a complete classification of all such forms within the range $\ell-m\le 25$.

\begin{proposition}\label{l:m1algebraic}
Suppose $\ell-m\leq 25$. 
Then the Miller form $g_{k,m}$ has a non-elliptic algebraic root if and only if $m=\ell-1$ and the form appears in the following table:
\begin{center}
\begin{tabular}{ll}
  Zero & Values of $k$\\
  \hline
  $\frac{1+\sqrt{-19}}{2}$&$442740, 442864, 442494, 442988, 442618, 442742$\\
  $\frac{1+\sqrt{-27}}{2}$&$6144372, 6144496, 6144126, 6144620, 6144250, 6144374$\\
  $\frac{1+\sqrt{-43}}{2}$&$442368372, 442368496, 442368126, 442368620, 442368250, 442368374$\\  
  $\frac{1+\sqrt{-67}}{2}$&$73598976372, 73598976496, 73598976126, 73598976620, 73598976250, 73598976374$\\  
  $\frac{1+\sqrt{-163}}{2}$&$131268706320384372,
                             131268706320384496,
                             131268706320384126$, \\
  &$131268706320384620, 131268706320384250, 131268706320384374$.  
\end{tabular}
\end{center}
\end{proposition}
\begin{proof}
If $z$ is algebraic zero of $g_{k,\ell-D}$, then \cite[Thm
1]{kohnen:eisenstein} implies that there exists $d>0$ such that
$j(z)$ is a root of its Faber polynomial, where $z=\sqrt{-d}$ or
$\frac{1+\sqrt{-d}}{2}$ depending on $-d\equiv 0,1\pmod{4}$. Thus the minimal polynomial of $j(z)$, the Hilbert Class Polynomial of $-d$ whose degree is the class number of the corresponding order in $\mathbb{Q}(\sqrt{-d})$, is an irreducible factor of the Faber polynomial of $g_{k,\ell-D}$. 

Our proof relies on two computations: (1) a computation of the Faber polynomials $F_{k',D}\in \mathbb{Z}[\ell][x]$ of $g_{k,\ell-D}$ symbolically for small
values of $D$ and (2) a database of all Hilbert Class Polynomials of degree at most $D$. We computed the former ourselves, and relied on \cite{watkins:classno} and \cite{klaise:orders} for a complete classification of orders of class number up to 100.

When $D=1$, the Faber polynomial of $g_{k,\ell-1}$ is 
$\displaystyle x+24\ell-744+\frac{2k'}{B_{k'}}$, where $B_{k'}$ is the $k'$-th Bernoulli number. The following table makes these values explicit. 
\begin{table}[h]
\centering

\begin{minipage}{0.45\textwidth}
\centering
\begin{tabular}{ll}
  $k'$ & Faber polynomial of $g_{k,\ell-1}$\\
  \hline
  0  & $x + 24\ell - 744$\\
  4  & $x + 24\ell - 984$\\
  6  & $x + 24\ell - 240$\\
\end{tabular}
\end{minipage}
\hfill
\begin{minipage}{0.45\textwidth}
\centering
\begin{tabular}{ll}
  $k'$ & Faber polynomial of $g_{k,\ell-1}$\\
  \hline
  8   & $x + 24\ell - 1224$\\
  10  & $x + 24\ell - 480$\\
  14  & $x + 24\ell - 720$\\
\end{tabular}
\end{minipage}

\end{table}

We conclude that we must seek the possible algebraic zeros among the orders with class number 1 in quadratic
imaginary fields, for which $d$ is in the set
\[ \{-3, -4, -7, -8, -11, -12, -16, -19, -27, -28, -43, -67, -163\}.\]
The table of algebraic roots follows from a computation of $\ell$ corresponding to these $j$-values. 

For every other discriminant $-d$ with Hilbert Class Polynomial $H(x)$ of degree $\leq D$ we checked that the $F_{k',D}(x,\ell)\not\equiv 0\pmod{H(x)}$, by verifying that the coefficients of the residue, which are polynomials in $\ell$, have no integral roots.
\end{proof}
We remark that the computation above runs exponentially slow, for $D=10,20,30,40,50$ requiring $2$ seconds, $1, 5, 23, 90$ minutes.

Finally, we conclude with a weaker computational result for
$25< D\leq 100$.
\begin{proposition}
As $T\to \infty$, for all values of $k=12\ell+k'<T$ except, possibly, for $O(\sqrt{T})$ values, the Miller form $g_{k,m}$ has no non-elliptic algebraic roots whenever $25<\ell-m\leq 100$. 
\end{proposition}
\begin{proof}
A non-elliptic algebraic root $z$ of $g_{k,m}$ yields a root
$j(z)$ of the Faber polynomial $F(x)$ of $g_{k,m}$. The Galois
closure of $\mathbb{Q}(j(z))/\mathbb{Q}$ has dihedral Galois
group so it suffices to verify that $F(x)$ is irreducible and
has non-dihedral Galois group. Fixing $D\leq 100$ and $k'$, the
Faber polynomial $F_{k',D}(x,\ell)\in \mathbb{Q}[\ell][x]$ is a family of
polynomials over $\mathbb{Q}[\ell]$, and we denote $G_{k',D}$ the Galois group  of $F_{k',D}(x,y)$ over $\mathbb{Q}(y)$. By the Hilbert Irreducibility Theorem \cite[\S9]{serre:mordell}, for all but $O(\sqrt{T})$ values of $\ell<T$, the Galois group of $F_{k',D}(x,\ell)$ is also $G_{k',D}$. 

For each $k'$ and each $D$, we found in Sage (at least) two primes $p$ and $q$ among the first 1000 primes such
that (1) $F_{k',D}(x,0)$ is irreducible mod $p$ and
(2)  $F_{k',D}(x,0)$ is a product of a linear term and an irreducible polynomial mod $q$. This implies that the Galois group of $F_{k',D}(x,\ell)$ is neither cyclic nor dihedral for a positive density of $\ell$ (whenever $\ell$ is a multiple of $p^aq^b$ for some $a,b$ larger than the power of $p$ and $q$ in the denominators of $F_{k',D}(x,\ell)$), which implies that $G_{k',\ell}$ is neither cyclic nor dihedral. We conclude that the Galois group of $F_{k',D}(x,\ell)$ is $G_{k',\ell}$ for all but $O(\sqrt{T})$ values of $\ell<T$, and the result follows.

\end{proof}


\providecommand{\bysame}{\leavevmode\hbox to3em{\hrulefill}\thinspace}
\providecommand{\MR}{\relax\ifhmode\unskip\space\fi MR }
\providecommand{\MRhref}[2]{%
  \href{http://www.ams.org/mathscinet-getitem?mr=#1}{#2}
}
\providecommand{\href}[2]{#2}

\end{document}